\documentclass[aos,preprint]{imsart}

\RequirePackage[OT1]{fontenc}
\RequirePackage{amsmath,amssymb,amsfonts,amsthm,bbm,mathrsfs,dsfont}
\RequirePackage[numbers,]{natbib} 
\RequirePackage[colorlinks,citecolor=blue,urlcolor=blue]{hyperref}
\usepackage{color,graphicx}

\startlocaldefs
\numberwithin{equation}{section}
\theoremstyle{plain}
\newtheorem{theorem}{Theorem}
\newtheorem{definition}{Definition}
\newtheorem*{definition*}{Definition}
\newtheorem{corollary}{Corollary}
\newtheorem{lemma}{Lemma}
\newtheorem{proposition}{Proposition}

\newtheorem*{assumption}{Assumption}
\newtheorem{remark}{Remark}
\theoremstyle{definition}
\newtheorem*{example}{Example}
\def\X{{\cal X}}
\def\A{{\cal A}}
\def\RR{\mathbb{R}}
\def\prodin{\prod_{i=1}^n}
\mathchardef\given="626A
\def\ra{\rightarrow}
\renewcommand{\P}{\mathbb{P}}
\newcommand{\F}{\mathbb{F}}
\renewcommand{\H}{\mathbb{H}}
\newcommand{\E}{\mathbb{E}}
\newcommand{\R}{\mathbb{R}}

\newcommand{\G}{\mathbb{G}}
\newcommand{\mA}{\mathcal{A}}
\newcommand{\mH}{\mathcal{H}}
\newcommand{\mF}{\mathcal{F}}
\newcommand{\mZ}{\mathcal{Z}}
\newcommand{\mC}{\mathcal{C}}
\newcommand{\mW}{\mathcal{W}}
\renewcommand{\L}{\mathcal{L}}
\newcommand{\mG}{\mathcal{G}}

\newcommand{\eps}{\varepsilon}

\newenvironment{enumerate*}%

\endlocaldefs

\begin{document}

\begin{frontmatter}
\title{Semiparametric Bayesian causal inference}
\runtitle{Bayesian causal inference}

\begin{aug}
\author{\fnms{Kolyan} \snm{Ray}\thanksref{t1}\ead[label=e1]{kolyan.ray@kcl.ac.uk}}
\and
\author{\fnms{Aad} \snm{van der Vaart}\thanksref{t1}\ead[label=e2]{avdvaart@math.leidenuniv.nl}}

\thankstext{t1}{The research leading to these results has received funding from the European Research Council under ERC Grant Agreement 320637.}
\runauthor{K. Ray and A.W. van der Vaart}

\affiliation{King's College London and Universiteit Leiden}

\address{Department of Mathematics\\
King's College London\\
Strand\\
London WC2R 2LS\\
United Kingdom\\
\printead{e1}}

\address{Mathematical Institute\\
Leiden University\\ 
P.O. Box 9512\\
2300 RA Leiden\\
Netherlands\\
\printead{e2}}

\end{aug}

\begin{abstract}
We develop a semiparametric Bayesian approach for estimating the mean response in a missing data
model with binary outcomes and a nonparametrically modelled propensity score. 
Equivalently, we estimate the causal effect of a treatment, correcting nonparametrically for confounding. 
We show that standard Gaussian process priors satisfy a semiparametric Bernstein--von Mises theorem under
smoothness conditions. We further propose a novel propensity score-dependent prior that provides
efficient inference under strictly weaker conditions. We also show that it is theoretically
preferable to model the covariate distribution with a Dirichlet process or Bayesian bootstrap,
rather than modelling its density.
\end{abstract}



\end{frontmatter}

\section{Introduction}
In many applications, one wishes to make inference concerning the causal effect of a treatment or condition. 
Examples include healthcare and assessing the impact
of public policies amongst many others. The available data are often observational rather
than the result of a carefully planned experiment or trial. The notion of ``causal'' then needs
to be carefully defined and the statistical analysis must take into account other possible
explanations for the observed outcomes.

A common framework for causal inference is the potential outcome setup \cite{imbens2015,robins1986}. In this
framework, every individual possesses two ``potential outcomes'', corresponding to the individual's
outcomes with and without treatment. The treatment effect, which we wish to estimate, is thus the
difference between these two potential outcomes. Since we only observe one out of each pair of
outcomes, and not the corresponding ``counterfactual'' outcome, we do not directly observe samples of
the treatment effect. Because in practice, particularly in observational studies, individuals are assigned
treatments in a biased manner, a simple comparison of actual cases (i.e.\ treated individuals) and controls may be misleading 
due to selection bias. A typical way to overcome this is to gather the values of covariate variables
that influence both outcome and treatment assignment (``confounders'') and apply a correction 
based on the ``propensity score'', which is the conditional probability that a subject is treated
as a function of the covariate values. Under the assumption that outcome and treatment assignment
are independent given the covariates, the causal effect of treatment can be identified from the
data. Popular estimation methods include ``propensity score matching'' 
\cite{rubin1978,rosenbaum1983} and ``double robust methods'' \cite{robins1986,robins1995,rotnitzky1995,seaman2018}.
In this paper we follow the approach of nonparametrically modelling the propensity score function 
and posing the estimation of the treatment effect as a problem of estimation of a functional on 
a semiparametric model \cite{BKRW,vandervaart1991,vandervaart1998}. Our methodological novelty is to follow a 
semiparametric Bayesian approach, putting nonparametric priors on the propensity score
and/or on the unknown response function and the covariate distribution, possibly incorporating
an initial estimator of the first function. 

For notational simplicity we in fact consider the missing data model which is mathematically
equivalent to observing one arm of the causal setup. The model is also standard and widely-studied
on its own in biostatistical applications, where response variables are frequently missing, and is a
template for a number of other models \cite{robins2008,vandervaart2014}. 
For a recent review on estimating an average treatment effect over a (sub)population, a problem that has received
considerable attention in the econometrics, statistics and epidemiology literatures, see
Athey et al. \cite{athey2017}.

Suppose that we observe $n$ i.i.d.\ copies $X_1,\dots,X_n$ of a random variable $X = (Z,R,RY)$,
where $R$ and $Y$ take values in the two-point set $\{0,1\}$ and are conditionally independent given
$Z$. We think of $Y$ as the outcome of a treatment and are interested in estimating its expected
value $\E Y$.  The problem is that the outcome $Y$ is observed only if the indicator variable $R$
takes value 1, as otherwise the third component of $X$ is equal to 0.  Whether the outcome is
observed or not may well be dependent on its value, which precludes taking a simple average of the
observed outcomes as an estimator for $\E Y$. The covariate $Z$ is collected to correct for this
problem; it is assumed to contain exactly the information that explains why the response $Y$ is not
observed except for purely random causes, so that the outcome $Y$ and missingness indicator $R$ are
conditionally independent given $Z$, i.e.\ the outcomes are \textit{missing at random} (relative to $Z$).  

The connection to causal inference is that we may think of $Y$ as a ``counterfactual" outcome
if a treatment were assigned ($R=1$) and its mean as ``half'' the treatment effect under the
assumption of unconfoundedness. More precisely, if $Y^1$ and $Y^0$ denote the potential
outcomes when treated or not treated, then in the causal model one would observe
$(Z, R, Y^1R, Y^0(1-R))$ and be interested in estimating $\E Y^1-\E Y^0$ under the assumption
that $Y^0,Y^1$ are conditionally independent of $R$ given $Z$. One can think of the
missing data problem as simplifying this to observing $(Z, R, Y^1R)$ and estimating $\E Y^1$.
To estimate the causal effect one could apply the missing data problem a second time,
to the data $(Z, R,  Y^0(1-R))$, to estimate $\E Y^0$, or do a simultaneous analysis. In the
nonparametric setup there will be no essential difference between the two.

The model for a single observation $X$ can be described by the distribution of $Z$ and the two
conditional distributions of $Y$ and $R$ given $Z$. In this paper we model these three components
nonparametrically. We investigate a Bayesian approach, putting a nonparametric prior on the three
components, in particular Gaussian process and Dirichlet process priors. We then consider the mean
response $\E Y$ as a functional of the three components and study the induced marginal posterior
distribution of $\E Y$ from a frequentist perspective. The aim is to derive conditions under
which this marginal posterior distribution satisfies a Bernstein--von Mises theorem in the
semiparametric sense, thus yielding recovery of the mean response at a $\sqrt n$-rate and 
asymptotic efficiency in the semiparametric sense. 

In recent years Bayesian approaches have become increasingly popular due to their excellent
empirical performance for such problems \cite{hill2011,heckman2014,taddy2014,zigler2014,hahn2016,hahn2017,alaa2017b,futoma2017,alaa2018}. However, despite their increasing use in practice, there have been few corresponding theoretical results. Indeed, early
work on semiparametric Bayesian approaches to this specific missing data problem produced negative
results, proving that many common classes of priors, or more generally likelihood-based procedures,
produce inconsistent estimates assuming no smoothness on the underlying parameters, see the results
and discussion in \cite{robins1997,ritov2014}. We attempt to shed light on this apparent 
gap between the excellent empirical performance observed in practice and the potentially
disastrous theoretical performance.

The structured nature of the model, with three parameters (response function, propensity score and covariate distribution),
 requires careful consideration of prior distributions. As the likelihood factorizes over the three parameters, 
choosing these a priori independent will lead to a product posterior.  
We show that this can lead to efficient estimation of $\E Y$, but only under unnecessarily harsh 
smoothness requirements on the parameters.  This is in agreement with the discussion in \cite{robins1997,ritov2014}, 
which applies to likelihood-based methods in general, including semiparametric maximum likelihood \cite{murphy}.  
Within our Bayesian setup it is possible to correct this (partly) by modelling the response function and 
propensity score as a priori dependent, thus allowing the components to share information, despite
the factorisation in the likelihood.
In particular, we propose a novel Gaussian process prior that incorporates an estimate of the propensity score function,
 and show that it performs efficiently under strictly weaker conditions than for standard product priors (see \cite{RaySzabo} for an empirical investigation). 
Unlike for these latter priors, extra regularity of the binary regression function can compensate for low regularity 
of the propensity score, that is one direction of so-called ``double robustness" \cite{rotnitzky1995,RobinsRotnitzky2001}. 
A related construction using Bayesian additive regression trees (BART) has been shown to work well empirically \cite{hahn2017}.  It can thus be both practically and theoretically advantageous to employ propensity score-dependent priors.

For the estimation of $\E Y$ at $\sqrt n$-rate, smoothness of the distribution of the covariate $Z$ is not
needed. In our main result, we therefore model this distribution by the standard nonparametric prior for a distribution: 
the Dirichlet process. In our concrete examples the prior modelling thus consists of a combination of Gaussian and Dirichlet processes.
 In the supplementary material we also consider modelling the covariate
density, for instance by an exponentiated Gaussian process. Our result seems to indicate that 
even when the smoothness of the density is modelled correctly, this approach 
can induce a non-vanishing bias in the posterior distribution of $\E Y$, 
an effect that becomes more pronounced with increasing covariate dimension.

The papers \cite{robins2008,robins2017} consider estimation of $\E Y$ under minimal
smoothness conditions on the parameters. Using estimating equations, the authors construct 
estimators that attain an optimal rate of convergence slower than $\sqrt n$ in cases where the
component parameters have low smoothness. Furthermore, they construct estimators
that attain a $\sqrt n$-rate under minimal smoothness conditions, less stringent than in earlier
literature, using higher order estimating equations. It is unclear whether similar results can be
obtained using a Bayesian approach. The constructions in the present paper can be
compared to the estimators obtainable for linear (or first order) estimating equations. It remains to
be seen whether Bayesian modelling is capable of performing the bias corrections necessary
to handle true parameters of low smoothness levels in a similar manner as higher
order estimating equations.

For smooth parametric models, the theoretical justification for posterior based inference is
provided by the Bernstein--von Mises theorem or property (hereafter BvM). This property says that as
the number of observations increases, the posterior distribution is approximately a Gaussian
distribution centered at an efficient estimator of the true parameter and with covariance equal to
the inverse Fisher information, see Chapter 10 of \cite{vandervaart1998}. While such a result does
not hold in full generality in infinite dimensions \cite{freedman1999}, semiparametric analogues can
establish the BvM property for the marginal posterior of a finite-dimensional parameter in the
presence of an infinite-dimensional nuisance parameter
\cite{castillo2012,rivoirard2012,bickel2012,castillo2015}. In such cases, care is required
in the choice of prior assigned to the nonparametric part, as oversmoothing may
induce a bias in the posterior distribution of the finite-dimensional parameter.

Our main results are two theorems for general priors on the response function and/or propensity score, 
followed by corollaries for Gaussian process priors. In both cases we combine these with a Dirichlet process prior
on the covariate distribution. While the first theorem is in the spirit of earlier work, it
is novel in its extension to a structured semiparametric model and its combination with the Dirichlet process, in both a modelling and a technical sense. 
The second theorem is innovative in its investigation of ``half of double-robustness'', as indicated in the
preceding, and by showing that incorporating a prior perturbation in the least favourable direction 
can remove potential bias from the posterior. The latter device takes care of the usual ``prior invariance condition''
and has consequences beyond the model in this paper.
 The corollaries for Gaussian process priors illustrate the conditions
of the main results, and give concrete examples of inference. In the supplementary material we
present a third theorem, which covers the case that the covariate density, rather than the distribution, is modelled,
which is again illustrated by Gaussian process priors.

An important consequence of the semiparametric BvM is that credible
sets for the functional are asymptotically confidence regions with the
same coverage level. The Bayesian approach thus automatically provides
access to uncertainty quantification once one can sample from the
posterior distribution. Obtaining confidence statements for average
treatment effects is a current area of research and there has been
recent progress in this direction, for example using random forests
and regression trees \cite{athey2016,wager2017}. Our results
show that Bayesian approaches can also yield valid frequentist
uncertainty quantification in this setting.

The paper is structured as follows. In Section~\ref{sec:model}, we provide a review of the model,
including the relevant semiparametric theory. Section~\ref{sec:results} contains the two main theorems
and their corollaries, with discussion in Section~\ref{sec:discussion} 
and the main proofs in Section~\ref{sec:proofs}. The remaining sections
are given in the supplement of the paper. Section~\ref{sec:general_prior_f} gives the third theorem, with
a joint prior on the propensity score, response function and covariate density. 
Technical results, auxiliary results and posterior contraction results are deferred to
Sections~\ref{SectionTechnicalResults}, \ref{sec:auxiliary_results} and~\ref{sec:contraction}, respectively. 

\subsection{Notation}
The notation $\lesssim$ denotes inequality up to a multiplicative constant that is fixed throughout and
$\lfloor x\rfloor$ is the largest integer strictly smaller than $x$. The symbol $\Psi$ is used
for the logistic function given by $\Psi(x)=1/(1+e^{-x})$. We abbreviate $\int f\, dP$ by
$Pf$. For probability densities $f$ and $g$ with respect to some dominating measure $\nu$,
$h(f,g) = (\int (f^{1/2}-g^{1/2})^2 d\nu)^{1/2}$ is the Hellinger distance, $K(f,g)$ is the
Kullback-Leibler divergence and $V(f,g) = \int (\log (f/g))^2\, dF$. We denote by $H^s=H^s([0,1]^d)$
and $C^s=C^s([0,1]^d)$ the $L^2$-Sobolev and H\"older spaces, respectively. For i.i.d. random variables
$X_1,\dots,X_n$ with common law $P$ the notations $\P_n [h]= n^{-1} \sum_{i=1}^n h(X_i)$ and
$\G_n[h]=\sqrt n(\P_n- Ph)$ are the empirical measure and process, respectively.  The notation
$\L(Z)$ denotes the law of a random element $Z$.  We often drop the index $n$ in the product
measure $P_\eta^n$, writing $P_\eta$, and write $P_0$ instead of $P_{\eta_0}$, where $\eta_0$
is the true parameter for the data generating distribution. The $\varepsilon$-covering number of a
set $\Theta$ for a semimetric $d$, denoted $N(\Theta,d,\varepsilon)$, is the minimal number of
$d$-balls of radius $\varepsilon$ needed to cover $\Theta$, and $N_{[]}(\Theta,d,\varepsilon)$ is
the minimal number of brackets of size $\varepsilon$ needed to cover a set of functions $\Theta$.

\section{Model details}\label{sec:model}
Recall that we observe i.i.d.\ copies $X_1,\dots,X_n$ of a random variable $X = (Z,R,RY)$, where $R$
and $Y$ take values in the two-point set $\{0,1\}$ and are conditionally independent given $Z$,
which itself takes values in a given measurable space $\mZ$.
Denote the full sample by $X^{(n)} = (X_1,\dots,X_n)$. This
model can be parameterized via the marginal distribution $F$ of $Z$ and the
conditional probabilities $a(z)^{-1}=P(R=1|Z=z)$, called the \textit{propensity score}, and
$b(z) = P(Y=1|Z=z)$, the regression of $Y$ on $Z$.
The distribution of an observation $X$ is thus fully described by the triple $(a,b,F)$. 
If $F$ has a density $f$, then we may also use the triple $(a,b,f)$.

For prior construction it will be useful to transform the parameters by a link function. Most smooth
maps from $\RR$ to $(0,1)$ may be used, but for definiteness we choose the logistic function $\Psi(t) = 1/(1+e^{-t})$, 
and consider the reparametrization
\begin{align}\label{eq:paramtrization}
\eta^a = \Psi^{-1}(1/a),\quad \quad  \eta^b = \Psi^{-1}(b),
\end{align}
and write $\eta = (\eta^a,\eta^b)$.
If a density $f$ of $Z$ exists, then we define in addition
$$\eta^f = \log f,$$
and write by a slight abuse of notation $\eta=(\eta^a,\eta^b,\eta^f)$. 

The density $p_{(a,b,f)} = p_\eta$ of $X$ can now be given as
\begin{align}\label{eq:likelihood_full}
p_\eta (x) 
=  \Bigl(\frac{1}{a(z)}\Bigr)^{r} \Bigl(1-\frac{1}{a(z)}\Bigr)^{1-r} b(z)^{ry} (1-b(z))^{r(1-y)}\, f(z) .
\end{align}
Note that this factorizes over the parameters. If the covariate is not assumed to  have
a density and $\eta=(\eta^a,\eta^b)$, we use the same notation $p_\eta$, but then
the factor $f(z)$ is understood to be 1, and the expression is the conditional density of $(R, RY)$ given $Z=z$.
Since $p_\eta$ factorizes over the three (or two) parameters, the log-likelihood based on $X^{(n)}$ separates as
\begin{equation}
\label{EqLogLikelihood}
\ell_n(\eta) = \sum_{i=1}^n\log p_{(a,b,f)}(X_i)=\ell_n^a (\eta^a) + \ell_n^b (\eta^b) + \ell_n^f (\eta^f),
\end{equation}
where each term is the logarithm of the factors involving only $a$ or $b$ or $f$,
and $\ell_n^f (\eta^f)$ is understood to be absent when existence of a density $f$ is not assumed.
The functional of interest is the \textit{mean response} $\E_\eta Y = \E_\eta b(Z)$, which 
can be expressed in the parameters as
\begin{equation*}
\chi (\eta) = \int b\, dF= \int \Psi(\eta^b)(z)\, e^{\eta^f(z)}\, dz,
\end{equation*}
where the second representation is available if $F$ has a density.

Estimators that are $\sqrt{n}$-consistent and asymptotically efficient for $\chi(\eta)$ have been
constructed using various methods, but only if $a$ or $b$ (or both) are sufficiently smooth. In the
present context, under the assumption that $a\in C^\alpha$ and $b\in C^\beta$, Robins et al.\
\cite{robins2017} have constructed estimators that are $\sqrt{n}$-consistent if
$(\alpha+\beta)/2 \geq d/4$, where $d$ is the dimension of the covariates. They have also shown that the latter condition is sharp: 
the minimax rate becomes slower than
$1/\sqrt{n}$ when $(\alpha+\beta)/2 < d/4$ (see \cite{robins2009}). The estimators in
\cite{robins2017} employ higher order estimating equations to obtain better control of the
bias. First-order estimators, based on linear estimators or semiparametric maximum likelihood,
have been shown to be $\sqrt n$-consistent only under the stronger condition
\begin{align}\label{eq:double_robust_smoothness}
\frac{\alpha}{2\alpha+d} + \frac{\beta}{2\beta+d} \geq \frac{1}{2},
\end{align}
see e.g. \cite{robins1995,rotnitzky1995}. 
In both cases the conditions show a trade-off between the smoothness levels of $a$ or $b$: higher
$\alpha$ permits lower $\beta$ and vice-versa.  This trade-off results from the multiplicative form
of the bias of linear or higher-order estimators.  So-called \emph{double robust} estimators are able to
exploit this structure, and work well if either $a$ or $b$ is sufficiently smooth.  (More generally,
it suffices that the parameters $a$ and $b$ can be estimated well enough, where the combined rates
are relevant.  The inequalities even remain valid with $\alpha=0$ or $\beta=0$ interpreted as the
existence of $\sqrt n$-consistent estimators of $a$ or $b$, as would be the case given a correctly
specified finite-dimensional model.)  We shall henceforth also assume that the parameters $a$ and
$b$ are contained in H\"older spaces $C^\alpha$ and $C^\beta$, respectively. See \cite{seaman2018}
for a recent discussion of double robustness.

For estimation of $\E Y$ at $\sqrt n$-rate the covariate density $f$ need not be smooth, which makes
sense intuitively, as the functional can be written as an integral relative to the corresponding
distribution $F$. (Counter to this intuition \cite{robins2009,robins2017} show this to be false for
optimal estimation at slower than $\sqrt n$-rate.) This may motivate modelling $F$
nonparametrically, in the Bayesian setting for instance with a Dirichlet process prior.

All these observations are valid only if the estimation problem is not affected by the parameters
$a$, $b$ or $f$ taking values on the boundary of their natural ranges. For simplicity we make the following assumption throughout.

\begin{assumption}
The true functions $1/a_0$ and $b_0$ are bounded away from 0 and 1 and $f_0$ is bounded away from 0 and $\infty$. 
\end{assumption}

\subsection{Semiparametric information and least favourable direction}
We finish by reviewing the tangent space and information distance of the model, which is well
known to play an important role in semiparametric estimation theory
\cite{BHHW,BKRW,vandervaart1991}, and enters the Bayesian derivations through the ``least favourable submodel''.  
(See \cite{castillo2012} or Chapter~12 of \cite{vandervaartbook2017} for general reviews in the context of Bayesian estimation.)

With regards to the parame\-tri\-zation \eqref{eq:paramtrization}, consider the one-dimensional
submodels $t\mapsto \eta_t$ induced by the paths 
$$\frac1{a_t}= \Psi(\eta^a + t\mathfrak{a}),\qquad 
b_t = \Psi (\eta^b + t\mathfrak{b}),\qquad 
dF_t = dF\,  e^{t\mathfrak{f}} \bigl({\textstyle\int} e^{t\mathfrak{f}}\,dF\bigr)^{-1}$$ 
for given directions $(\mathfrak{a},\mathfrak{b},\mathfrak{f})$ with $\int \mathfrak{f}\,dF = 0$, 
and given ``starting'' point $\eta=\eta_0$.
Inserting these paths in the likelihood \eqref{eq:likelihood_full},
and computing the derivative ${\tfrac{d}{dt}}_{|t=0} \log p_{\eta_t}(x)$ of the log likelihood, 
we obtain the ``score function'' at $\eta=\eta_0$ in the direction $(\mathfrak{a},\mathfrak{b},\mathfrak{f})$.
This can be easily computed to be the sum of the score functions when varying the three parameters
separately, which are given by 
\begin{equation*}
\begin{split}
& B_\eta^a \mathfrak{a}(X) = (R-\tfrac{1}{a(Z)}) \mathfrak{a}(Z),\\
& B_\eta^b \mathfrak{b}(X) = R(Y-b(Z))\mathfrak{b}(Z),\\
& B_\eta^f \mathfrak{f} (X) = \mathfrak{f}(Z).
\end{split}
\end{equation*}
The operators $B_\eta^a$, $B_\eta^b$, $B_\eta^f$ are the \textit{score operators} for the three
parameters. The overall score $B_\eta (\mathfrak{a},\mathfrak{b},\mathfrak{f})(X)$ when perturbing the three
parameters simultaneously is the sum of the three terms in the previous display.  The
\emph{efficient influence function} of the functional $\chi$ at the point $\eta$ is known to take the form 
(see Example~25.43 of \cite{vandervaart1998} with $\dot\chi_Q(y)$ the current $y-\chi(\eta)$ and
$\phi(y,0)$ the current $(R,Z)$, or the derivation below)
\begin{equation*}\label{eq:infl_fcn}
\widetilde{\chi}_{\eta} (X) = Ra(Z)(Y-b(Z)) + b(Z) - \chi (\eta).
\end{equation*}
We can verify that this is the correct formula by verifying that this function has the two properties defining an
efficient influence function (\cite{vandervaart1998}, page~426). 
First, the derivative at $t=0$ of the functional along a path $t\mapsto\eta_t=(a_t,b_t,f_t)$
as previously, is the inner product of the influence function with the score function of that path:
${\tfrac{d}{dt}}_{|t=0} \chi(\eta_t) = P_\eta \widetilde{\chi}_\eta(X) B_\eta(\mathfrak{a},\mathfrak{b},\mathfrak{f})(X)$
for every path $t \mapsto p_{\eta_t}$ of the above form. 
Second, the function $\widetilde{\chi}_{\eta}$ is contained in the closed linear span of the set of
all score functions.  Indeed, in the present case we have, for all $x$,
\begin{equation}
\label{EqEIF}
\widetilde{\chi}_\eta(x)= B_\eta \xi_\eta (x)=B_\eta^ba(x)+B_\eta^f\bigl(b-\textstyle{\int} b\,dF\bigr)(x),
\end{equation}
where $\xi_\eta$ is the \textit{least favourable direction} given by
\begin{equation*}
\xi_\eta = (0,\xi_\eta^b,\xi_\eta^f) =  \bigl(0,a,b-\textstyle{\int} b\,dF \bigr).
\end{equation*}
The function $\xi_\eta$ is the score function for the submodel $t\mapsto \eta_t$ corresponding to
the perturbations in the directions of $(0,a, b-\int b\,dF)$ on $(a,b,F)$. The latter submodel is
called \textit{least favourable}, since $t\mapsto p_{\eta_t}$ has the smallest information about the
functional of interest at $t=0$.  According to semiparametric theory (e.g.\ Chapter 25 of
\cite{vandervaart1998}, in particular formula (25.22)) a sequence of estimators $\widehat{\chi}_n=\widehat{\chi}_n (X^{(n)})$ is
asymptotically efficient for estimating $\chi(\eta)$ at the true parameter $\eta_0$ if and only if
\begin{equation}
\label{EqEffcieintEstimator}
\widehat{\chi}_n = \chi(\eta_0) + \frac{1}{n} \sum_{i=1}^n \widetilde{\chi}_{\eta_0}(X_i)+o_{P_{\eta_0}}(n^{-1/2}).
\end{equation}
The sequence $\sqrt n\bigl(\widehat{\chi}_n -\chi(\eta_0)\bigr)$ is then asymptotically normal with
mean zero and variance $P_{\eta_0} \widetilde{\chi}_{\eta_0}^2$, which is the smallest possible in a local
minimax sense.  

For a direction $v = (\mathfrak{a},\mathfrak{b},\mathfrak{f})$, the \emph{information norm} corresponding to the score
operator (or LAN norm in the language of \cite{castillo2012,rivoirard2012,castillo2015}) equals
\begin{equation*}
\begin{split}
\| v\|_\eta^2  & := P_{\eta} [(B_{\eta}v)]^2
= \int \Bigl[ \frac{1}{a}\bigl(1-\frac{1}{a}\bigr) \mathfrak{a}^2 + \frac{b(1-b)}{a}  \mathfrak{b}^2 + (\mathfrak{f}-F\mathfrak{f}) ^2 \Bigr]\,dF\\
& =:\|\mathfrak{a}\|_{a}^2 + \|\mathfrak{b}\|_b^2 + \|\mathfrak{f}\|_F^2 .
\end{split}
\end{equation*}
It may be noted that the three components of the score operator are orthogonal, which is 
a consequence of the factorization of the likelihood.
The minimal asymptotic variance $P_{\eta_0} \widetilde{\chi}_{\eta_0}^2$ for estimating $\chi(\eta)$ can
be written in terms of the information norm as
\begin{equation}
\begin{split}
\|\xi_{\eta_0}\|_{\eta_0}^2 &= P_{\eta_0} (B_{\eta_0} \xi_{\eta_0})^2 
= P_{\eta_0}\widetilde{\chi}_{\eta_0}^2 \label{EqEfficientVariance}\\
&= \int a_0b_0(1-b_0)\,dF_0 + \int b_0^2 \,dF_0 - \chi(\eta_0)^2.
\end{split}
\end{equation}

\section{Results}\label{sec:results}
We put a prior probability distribution $\Pi$ on the parameter  $(\eta^a,\eta^b,F)$ or $\eta=(\eta^a,\eta^b,\eta^f)$,
and consider the posterior distribution $\Pi(\cdot|X^{(n)})$ based on the observation $X^{(n)}=(X_1,\dots,X_n)$. 
This induces posterior distributions on all measurable functions of $\eta$, 
including the functional of interest $\chi(\eta)$. 

We write $\mathcal{L}_\Pi(\sqrt{n}(\chi(\eta) -\widehat{\chi}_n) |X^{(n)})$ for the marginal
posterior distribution of $\sqrt{n}(\chi(\eta)-\widehat{\chi}_n)$, where $\widehat{\chi}_n$ is any
random sequence satisfying \eqref{EqEffcieintEstimator}.
We shall be interested in proving that this distribution asymptotically looks like a centered normal
distribution with variance $\|\xi_{\eta_0}\|_{\eta_0}^2$.  For a precise statement of this approximation,
let $d_{BL}$ be the bounded Lipschitz distance on probability distributions on $\RR$
(see Chapter 11 of \cite{dudley2002}).

\begin{definition}
\label{DefBvM}
Let $X^{(n)}=(X_1,\dots,X_n)$ be i.i.d. observations with $X_i = (Z_i,R_i,R_iY_i)$ arising from the density $p_{\eta_0}$ in \eqref{eq:likelihood_full}, whose distribution we denote by $P_0 = P_{\eta_0}$. We say that the posterior satisfies the \textit{semiparametric Bernstein--von Mises (BvM)} if,
for $\widehat{\chi}_n$ satisfying \eqref{EqEffcieintEstimator} and $\|\xi_{\eta_0}\|_{\eta_0}$ given by \eqref{EqEfficientVariance}, as $n\rightarrow \infty$, 
\begin{align*}
d_{BL} \Bigl( \mathcal{L}_\Pi\bigl(\sqrt{n}(\chi(\eta) -\widehat{\chi}_n\bigr) |X^{(n)}), N(0,\|\xi_{\eta_0}\|_{\eta_0}^2 ) \Bigr) \rightarrow^{P_0} 0.
\end{align*}
\end{definition}

In Sections~\ref{SubsectionGeneralabDir} and~\ref{SectionPropensityScorePriors} we present two general results for priors 
on the parameters $(a, b)$, combined with an independent Dirichlet process prior on $F$. In Section~\ref{SubsectionGeneralabDir} the prior on the pair $(a,b)$ is general, whereas 
in Section~\ref{SectionPropensityScorePriors} we construct a prior on $b$ using an estimator of the propensity score $1/a$, thus linking the two parameters.
Following these general results we specialize to Gaussian process priors and obtain concrete results in Section~\ref{SectionGaussianPriors}.

An alternative to using the Dirichlet process on $F$ is to put a prior on the triple $(a,b,f)$, for $f$ a density of $F$.
A general result can be found in Section~\ref{sec:general_prior_f} below, but it requires stronger conditions for the
BvM theorem to hold. Putting a prior on $f$ introduces the additional bias term\eqref{eq:no_bias_f}, whose vanishing becomes more restrictive as 
the covariate dimension increases and can be problematic in even moderate dimensions. Thus it appears preferable to directly model the distribution $F$.

\subsection{Posterior distribution relative to Dirichlet process prior}
Since the covariates $Z_1,\ldots, Z_n$ are fully observed and the functional of interest $\chi(\eta)$ is an integral
relative to their distribution $F$, intuitively the estimation problem should not depend too much on properties of the covariate distribution.
For $\sqrt n$-estimation this intuition is shown to be correct in \cite{robins2017}. In our Bayesian setup
this suggests to put a prior on $F$ that does not limit this distribution.

The standard ``nonparametric prior'' on the set of probability distributions on a (Polish) sample space is the
Dirichlet process prior \cite{Ferguson1974}. This distribution is characterized by a base
measure $\nu$, which can be any finite measure on the sample space. It is well known that in the model
consisting of sampling $F$ from the Dirichlet process prior and next sampling observations
$Z_1,\ldots, Z_n$ from $F$, the posterior distribution of $F$ given $Z_1,\ldots, Z_n$ is again a
Dirichlet process with updated base measure $\nu+n\F_n$, where $\F_n$ is the empirical
distribution of $Z_1,\ldots, Z_n$. (For full definitions and properties, see the review in Chapter 4 of
\cite{vandervaartbook2017}.)

We utilize the Dirichlet process prior on $F$ together with an independent
prior on the remaining parameters $(a,b)$, constructed from a prior on $(\eta^a,\eta^b)$  using the logistic link
function \eqref{eq:paramtrization}. Because the Dirichlet process prior does not give probability one
to a dominated set of measures $F$, the resulting posterior distribution
of $(a,b,F)$ cannot be derived using Bayes's formula. However, we can obtain
a representation as follows. The parameters and the data are generated through the hierarchical scheme:
\begin{itemize}
\item $F\sim DP (\nu)$ independent from $\eta=(a,b)\sim \Pi$.
\item Given $(F,a,b)$ the covariates $Z_1,\ldots,Z_n$ are i.i.d.\ $F$.
\item Given $(F,a,b,Z_1,\ldots,Z_n)$ the pairs $(R_i,Y_i)$ are independent
from products of binomial distributions with success probabilities $1/a(Z_i)$ and $b(Z_i)$.
\item The observations are $X^{(n)}=(X_1,\ldots, X_n)$ with $X_i=(Z_i, R_i, R_iY_i)$.
\end{itemize}
From this scheme it follows that $F$ and $(R^{(n)}, Y^{(n)})$ are independent given
$(Z^{(n)}, a,b)$, and also that $F$ and $(a,b)$ are conditionally independent given $X^{(n)}$.
We can then conclude that the posterior distribution of $F$ given $X^{(n)}$ is the same
as the posterior distribution of $F$ given $Z^{(n)}$, which is the $DP(\nu+n\F_n)$ distribution.
Furthermore, the posterior distribution of $(a,b)$ given $(F, X^{(n)})$ can be derived
by Bayes's rule from the binomial likelihood of $(R^{(n)}, R^{(n)}Y^{(n)})$ given $Z^{(n)}$,
which is dominated. Thus the posterior distribution is given by
\begin{equation}\label{EqDirichletPosterior}
\begin{split}
&\Pi\bigl((a,b)\in A, F\in B|X^{(n)}\bigr)\\
&\qquad\qquad=\int_B \frac{\int_A  \prodin p_{(a,b)}(R_i,R_iY_i\given Z_i)\,d\Pi(a,b)}
{\int \prodin p_{(a,b)}(R_i,R_iY_i\given Z_i)\,d\Pi(a,b)}\,d\Pi(F\given Z^{(n)}),
\end{split}
\end{equation}
where $p_{(a,b)}$ is the conditional density of $(R,RY)$ given $Z$, given by
\eqref{eq:likelihood_full} with $f$ deleted or taken equal to 1, and 
$\Pi(F\in\cdot\given Z^{(n)})$ is the $DP(\nu+n\F_n)$-distribution.
This formula remains valid if $\nu=0$, which yields the Bayesian bootstrap, 
see Chapter 4.7 of \cite{vandervaartbook2017}, and is also covered
in the theorems below. We suspect that the theorems extend to other exchangeable bootstrap
processes, as considered in \cite{PraestgaardWellner} (see \cite{vdVWPreservation}, Section~3.7.2).

\subsection{General prior on $(a, b)$ and Dirichlet process prior on $F$}
\label{SubsectionGeneralabDir}
Define $\eta_t(\eta)=\eta_t(\eta; n,\xi_{\eta_0})$ to be a perturbation of $\eta=(\eta^a,\eta^b)$ in the least favourable direction, 
restricted to the components corresponding to $a$ and $b$:
\begin{equation}\label{eq:LFD_F_no_den}
\eta_t(\eta) 
= \Bigl(\eta^a, \eta^b - \frac{t}{\sqrt n}\xi_{\eta_0}^b\Bigr).
\end{equation}

\begin{theorem}\label{thm:general_F}
Consider a prior $\Pi$ consisting of an arbitrary prior on $\eta = (\eta^a,\eta^b)$ and an independent
Dirichlet process prior on $F$. 
Assume that there exist measurable sets $\mH_n$ of functions $\eta=(\eta^a,\eta^b)$ satisfying
\begin{align}
\Pi(\eta\in \mH_n \given X^{(n)}) &\rightarrow^{P_0} 1,\label{EqPriorConcentration2}\\
 \sup_{b=\Psi(\eta^b):\eta\in \mH_n} \|b-b_0\|_{L^2(F_0)}  &\rightarrow 0,\label{EqConsistencyb2}\\
\sup_{b=\Psi(\eta^b): \eta \in \mH_n} |\G_n[b-b_0]| &\rightarrow^{P_0} 0 \label{eq:prob_sup2}.
\end{align}
If for the path $\eta_t(\eta)$ in \eqref{eq:LFD_F_no_den} and every $t$,
\begin{align}\label{eq:prior_shift_cond2}
\frac{\int_{\mH_n} \prodin p_{\eta_t(\eta)}(R_i,R_iY_i\given Z_i) \,  d\Pi(\eta)}
{\int_{\mH_n} \prodin p_{\eta}(R_i,R_iY_i\given Z_i)\, d\Pi(\eta)} \rightarrow^{P_0}1,
\end{align}
then the posterior distribution \eqref {EqDirichletPosterior} satisfies the BvM theorem.
\end{theorem}

Conditions \eqref{EqPriorConcentration2}--\eqref{eq:prob_sup2} permit to control the remainder terms
in an expansion of the likelihood. They require that the posterior distribution of $b$ concentrates on 
shrinking neighbourhoods about the true parameter $b_0$ (with no similar requirement for $a$),  and hence
mostly require consistency.

The uniformity in $b$ required in \eqref{eq:prob_sup2} is unpleasant, as it will typically require
that the class of $b$ supported by the posterior distribution is not unduly large.  The condition is
linked to using the likelihood and similar conditions arise in maximum likelihood based estimation
procedures, although \eqref{eq:prob_sup2} seems significantly weaker, as the uniformity is required only on
the essential support of the posterior distribution, which might be much smaller than the full
parameter space.  The use of estimating equations can avoid uniformity
conditions by sample splitting \cite{robins2017}. In the Bayesian framework one might
similarly base posterior distributions of different parameters
on given subsamples, but this is unnatural so that we do not pursue this
route here.

Under \eqref{EqConsistencyb2} a sufficient condition for \eqref{eq:prob_sup2} is that the class of functions $b$ in
the condition is contained in a fixed $F_0$-Donsker class (see Lemma~3.3.5 of
\cite{vandervaart1996}). In particular, it suffices that the posterior concentrates on a bounded set
in $H^s$ for $s>d/2$. While this condition is easy to establish for certain priors, such as uniform
wavelet priors \cite{gine2011}, for the Gaussian process priors considered below we
employ relatively complicated arguments using metric entropy bounds to verify the condition. 

Condition \eqref{eq:prior_shift_cond2} measures the invariance 
of the prior for the full nuisance parameter under a shift in the least favourable direction $\xi_{\eta_0}^b$.
It is a structural condition on the combination of prior and model, and if not satisfied may
destroy the $\sqrt n$-rate in the BvM theorem (see \cite{castillo2012}
or \cite{vandervaartbook2017} for further discussion). Although we shall verify the condition
for several priors of interest below, this condition may impose smoothness
conditions on the parameters, and prevent so-called ``double robustness''. We 
shall remove this condition for special priors in Theorem~\ref{thm:dep_prior_gen_EB} below.

The invariance involves the component $\xi_0^b$ only, and not the other nonzero component
$\xi_0^f$ of the least favourable direction. In contrast, in Theorem~\ref{thm:general_f}, which puts a
prior on the covariate density $f$, the invariance involves the full least favourable direction
(see \eqref{eq:LFD_f_den}). Intuitively, the Dirichlet process is a fully nonparametric prior
that never causes this type of bias.

Since $\xi_{\eta_0}^b = a_0$, Theorem~\ref{thm:general_F} implicitly requires conditions on $a_0$
through \eqref{eq:prior_shift_cond2}, even though $a$ does not appear in the functional
$\chi(\eta)$. Such conditions become explicit for concrete priors below.

\begin{remark}
If the quotient on the left side of \eqref{eq:prior_shift_cond2} is asymptotic to $e^{\mu_nt}(1+o_{P_0}(1))$
for some possibly random sequence of real numbers $\mu_n$, then
the assertion of the BvM theorem is still true, but the
normal approximation $N(0,\|\xi_{\eta_0}\|_{\eta_0}^2 ) $ must be replaced
by $N(\mu_n,\|\xi_{\eta_0}\|_{\eta_0}^2 )$. 
See \cite{castillo2015,rivoirard2012} for further discussion.
The same is true for all other results in the following.
\end{remark}

\begin{remark}
If the supremum in \eqref{eq:prob_sup2}, or similar variables below, is not measurable, then we interpret this
statement in terms of outer probability. 
\end{remark}

Formula \eqref{EqDirichletPosterior} shows that a draw from the posterior distribution of the
functional of interest $\chi(\eta)=\int b\,dF$ is obtained by independently drawing $b$ from its
posterior distribution and $F$ from the $DP(\nu+n\F_n)$-distribution, and next forming the integral
$\int b\,dF$. The posterior distribution of $b$ is constructed from the conditional likelihood of
$(R^{(n)}, R^{(n)}Y^{(n)})$ given $Z^{(n)}$ without involving $F$ or its prior
distribution. Instead of a Bayesian-motivated or bootstrap type choice for $F$, which requires
randomization given $Z^{(n)}$, one could also directly plug in an estimator of $F$ based on
$Z^{(n)}$ and randomize only $b$ from its posterior distribution. The empirical distribution $\F_n$ is an obvious choice.
The proof of Theorem~\ref{thm:general_F} suggests that for this choice, under the conditions of the theorem,
\begin{align*}
d_{BL} \Bigl( \mathcal{L}_\Pi\bigl(\sqrt{n}(\chi(\eta) - \widehat{\chi}_n)  |X^{(n)}\bigr), N(0,\|\xi_{\eta_0}^{b_0}\|_{b_0}^2 ) \Bigr) \rightarrow^{P_0} 0.
\end{align*}
Compared to the BvM theorem this suggests a normal approximation with
the same centering, but a smaller variance, since the variance in the BvM theorem is the sum
$\|\xi_{\eta_0}^{b_0}\|_{b_0}^2+ \|\xi_{\eta_0}^{f_0}\|_{F_0}^2$. The lack of
posterior randomization of $F$ thus results in an underestimation of the asymptotic variance. Using
credible sets resulting from this `posterior' would give overconfident (wrong) uncertainty quantification. 
Since our focus is on the Bayesian approach, we do not purse such generalizations further.

\subsection{Propensity score-dependent priors}
\label{SectionPropensityScorePriors}
To reduce unnecessary regularity conditions, it can be useful to use 
a preliminary estimate $\hat{a}_n$ of the inverse propensity score
\cite{robins2017,robins1995,rotnitzky1995}. In a Bayesian setting, \cite{hahn2017} suggest adding an 
estimate of the propensity score evaluated at the data as an additional covariate when using BART for causal inference \cite{hill2011}.
In this section we employ preliminary
estimators $\hat a_n$ to augment the prior on $b$ with the aim of weakening the
conditions required for a semiparametric BvM.

Suppose we have a sequence of estimators $\hat{a}_n$ of the inverse propensity score satisfying
\begin{align}\label{eq:estimator_prop}
\|\hat{a}_n-a_0\|_{L^2(F_0)} = O_{P_0}(\rho_n),
\end{align}
for some sequence $\rho_n \rightarrow 0$. Since the propensity score is just a (binary)
regression function of $R$ onto $Z$, standard (adaptive) smoothing estimators satisfy this condition
with rate $\rho_n=n^{-{\alpha}/({2\alpha+d})}$ if the propensity score is assumed to
be contained in $C^\alpha([0,1]^d)$, which is the minimax rate over this space 
(note that $\hat a_n-a_0=\hat a_n a_0(1/a_0-1/\hat a_n)$ will attain at least the rate of an estimator
of the propensity score $1/a_0$ itself).
Consider the following prior on $b$:
\begin{equation}\label{eq:prior_dep_EB}
b(z) = \Psi\bigl(W_z^b + \lambda \hat{a}_n(z)\bigr),
\end{equation}
where $W^b$ is a continuous stochastic process independent of the random variable
$\lambda$, which follows a prior $N(0,\sigma_n^2)$ distribution for given variance $\sigma_n^2$ 
(potentially varying with $n$, but fixed is allowed). The additional parameter $\lambda$ has the role
of making the prior link between the parameters $b$ and $a$ flexible; the variance $\sigma_n^2$  will
be required not too small below.

We assume that $\hat{a}_n$ is based on observations that are independent of $X_1,\ldots,X_n$, the observations used in the
likelihood to obtain the posterior distribution. Otherwise, the prior \eqref{eq:prior_dep_EB} becomes
data-dependent, which significantly complicates the technical analysis. This independence seems, however, unnecessary in practice. 
The analogous prior to \eqref{eq:prior_dep_EB} for a continuous regression model is investigated
 empirically in the companion paper \cite{RaySzabo}, 
where it performs well when $1/\hat a_n$ is trained on the same data as the posterior.

We may think of $\hat a_n$ as a degenerate prior on $a$, and then 
by the factorization of the likelihood the part of the likelihood involving 
$a$ cancels from the posterior distribution
\eqref {EqDirichletPosterior} if marginalized to $(b,F)$ (and hence $\chi(\eta)$).
Of course the same will happen if we assign an independent prior to $a$.
Thus in both cases it is unnecessary to further discuss a prior on $a$.


\begin{theorem}\label{thm:dep_prior_gen_EB}
Given independent estimators $\hat{a}_n$ satisfying \eqref{eq:estimator_prop} and having
$\|\hat{a}_n\|_\infty = O_{P_0}(1)$,
 consider the prior \eqref{eq:prior_dep_EB} for $b$ with
the stochastic process $W^b$ and random variable $\lambda\sim N(0,\sigma_n^2)$ independent, and 
assign $F$ an independent Dirichlet process prior. 
Assume that there exist measurable sets $\mH_n^b$ of functions satisfying, for every $t\in\R$ and
some numbers $u_n, \varepsilon_n^b\rightarrow 0$,
\begin{align}
\Pi\bigl(\lambda: |\lambda|\le u_n\sigma_n^2\sqrt{n}|X^{(n)}\bigr)&\rightarrow^{P_0} 1,
\label{eq:dep_prior_lambda_cond}\\
\Pi \bigl( (w,\lambda): w+(\lambda + tn^{-1/2})\hat{a}_n\in \mH_n^b|X^{(n)}\bigr)&\rightarrow^{P_0} 1,
\label{EqPriorConcentration3}\\
\sup_{b=\Psi(\eta^b):\eta^b\in \mH_n^b} \|b-b_0\|_{L^2(F_0)}  &\le \varepsilon_n^b,\label{EqConsistencyb3}\\
\sup_{b=\Psi(\eta^b): \eta^b \in \mH_n^b} \bigl|\G_n[b-b_0]\bigr| &\rightarrow^{P_0} 0 \label{eq:prob_sup3}.
\end{align}
If $n\sigma_n^2\rightarrow \infty$ and $\sqrt{n}\rho_n\varepsilon_n^b \rightarrow 0$, then
the posterior distribution satisfies the semiparametric BvM theorem.
\end{theorem}

The advantage of this theorem over Theorem~\ref{thm:general_F} is that \eqref{eq:prior_shift_cond2} does
not appear in its conditions. (The theorem adds \eqref{eq:dep_prior_lambda_cond} and
\eqref{EqPriorConcentration3}, but these are relatively mild.) As noted above, condition
\eqref{eq:prior_shift_cond2} requires a certain invariance of the prior of $b$ in the the least
favourable direction $\xi_{\eta_0}^b = a_0$, and typically leads to smoothness requirements on
$a$. In contrast we show below that Theorem~\ref{thm:dep_prior_gen_EB} can yield the BvM theorem for
propensity scores $1/a$ of arbitrarily low regularity. Thus the theorem is able to achieve
what could be named \emph{single robustness}. Whether  ``double robustness'', the ability of also
handling response functions $b$ of arbitrarily low smoothness, is also achieved remains unclear.
Specifically, we have not been able to verify condition \eqref{eq:prob_sup3} 
without assuming that the smoothness of $b$ is above the usual threshold ($d/2$  in $d$ dimensions).

The single robustness is achieved by perturbing the prior process for $b$ in the
least favourable direction using the auxiliary variable $\lambda$.
Since the least favourable direction $a_0$ is unknown, this is 
replaced with an estimate $\hat{a}_n$. 

Condition \eqref{eq:dep_prior_lambda_cond} puts a lower bound on the variability of the perturbation,
i.e.\ on the standard deviation $\sigma_n$ of $\lambda$. An easy method to ascertain this condition
is to show that the prior mass of the set $\lambda$ in the left side is exponentially small and next
invoke Lemma~\ref{lem:small_prior_prob}.  Specifically, by the univariate Gaussian tail bound the
prior mass of $\{\lambda: |\lambda|> u_n\sigma_n^2\sqrt{n}\}$ is bounded above by
$e^{-u_n^2\sigma_n^2n/2}$.  If the Kullback-Leibler neighbourhood in
Lemma~\ref{lem:small_prior_prob} has prior probability at least $e^{-n(\varepsilon_n^b)^2}$, then
the lemma gives the sufficient condition $u_n^2\sigma_n^2\gtrsim (\varepsilon_n^b)^2$ for
\eqref{eq:dep_prior_lambda_cond}, i.e.\ $\sigma_n\gg \varepsilon_n^b$.

\subsection{Specialization to Gaussian process priors}
\label{SectionGaussianPriors}
In this section we specialize Theorems~\ref{thm:general_F} and \ref{thm:dep_prior_gen_EB} to Gaussian
process priors. In all examples the priors on the three parameters $a$,
$b$ and $F$ are independent.  Since $a$ does not appear in $\chi(\eta)$ and the likelihood
\eqref{eq:likelihood_full} factorizes over $a$, $b$ and $F$, the $a$ terms cancel from the marginal
posterior distribution of $\chi(\eta)$. Thus the prior on $a$ is irrelevant, and it is not necessary
to consider it.

For simplicity we take the covariate space to be the unit cube $\mZ=[0,1]^d$.
Given a mean-zero Gaussian process $W^b= (W_z^b: z\in [0,1]^d)$, 
we consider both the propensity score-dependent prior for $b$ given by \eqref{eq:prior_dep_EB} and the more simple prior
\begin{align}\label{eq:prior_b}
b(z) &= \Psi(W_z^b ).
\end{align}
There are a great variety of Gaussian processes, and their success in
nonparametric estimation is known to depend on their sample smoothness, as measured through their small
ball probability (see \cite{vandervaart2008,vdVvZRescaling,vdVvZGamma,vdVvZGaussianLearning}).
We derive a proposition on general Gaussian processes and consider the following specific examples. 

\begin{example}[Riemann-Liouville]
In dimension $d=1$, the  \textit{Riemann-Liouville} process released at zero of regularity $\bar{\beta}>0$ is defined by
\begin{align}\label{eq:RL_process}
W^b_z= \sum_{k=0}^{\lfloor \bar{\beta}\rfloor +1} g_k z^k + \int_0^z (z-s)^{\bar{\beta}-1/2}\, dB_s, \qquad z\in[0,1],
\end{align}
where the $(g_k)$ are i.i.d.\ standard normal random variables and $B$ is an independent Brownian motion. 
This process is appropriate for nonparametric  modelling of $C^{\bar{\beta}}([0,1])$-functions. We shall investigate
the effect of the smoothness parameter $\bar\beta$ on the BvM theorem. 
\end{example}

\begin{example}[Gaussian series]
Another commonly used Gaussian process prior consists of a finite series expansion with Gaussian coefficients. 
Let $\{\psi_{jk}: j\geq 1, k =0,\dots,2^{jd}-1\}$ denote a sufficiently regular boundary-adapted Daubechies wavelet basis 
of $L^2([0,1]^d)$. We assume it is regular enough for the decay of the wavelet coefficients to characterize all the relevant Besov $B_{\infty\infty}^s$-norms
(which are equal to the $C^s$-H\"older norms for $s\not\in\mathbb{N}$. For details on such wavelets and Besov spaces, see Chapter 4.3 of \cite{gine2016}.)
Consider the  prior 
\begin{equation}\label{eq:series_prior}
W_z^b = \sum_{j=1}^{J_{\bar{\beta}}} \sum_{k=0}^{2^{jd}-1} \sigma_j g_{jk} \psi_{jk}(z), \quad \quad g_{jk}\sim^{iid} N(0,1),
\end{equation}
where $2^{J_{\bar{\beta}}} \sim n^{1/(2\bar{\beta}+d)}$, which tends to infinity with $n$, 
is the optimal dimension of a finite-dimensional model if the true parameter is known to be 
$\bar{\beta}$-smooth and $\sigma_j = 2^{-j(r+d/2)}$ for $r \geq 0$. 
Since we wish to perform the prior regularization via the truncation level $J_{\bar{\beta}}$ rather than the scaling coefficients $\sigma_j$, 
we restrict to considering $r \leq \beta \wedge \bar{\beta}$, for instance $r=0$.
\end{example}

We can view both processes as Borel-measurable maps in the Banach space $C([0,1]^d)$, equipped
with the uniform norm $\|\cdot\|_\infty$. In the following proposition we consider a general zero-mean 
Gaussian process of this type. Such a process determines a so-called reproducing kernel Hilbert space (RKHS) $(\H^b,\|\cdot\|_{\H^b})$, 
and a  ``concentration function'' at $\eta_0^b$, defined as, for $\varepsilon>0$,
\begin{align}
\label{EqConcentrationFunction}
\phi_{\eta_0^b} (\varepsilon) 
= \inf_{h\in \H^b: \|h-\eta_0^b\|_\infty< \varepsilon} \|h\|_{\H^b}^2 - \log P(\|W^b\|_\infty < \varepsilon).
\end{align}
For standard statistical models, the posterior contraction rate $\varepsilon_n^b$ for such a Gaussian process prior
is linked to the solution of the equation 
\begin{align}
\phi_{\eta_0^b} (\varepsilon_n^b) \sim n (\varepsilon_n^b)^2.
\label{eq:gaus_conc_rate}
\end{align}
For details see \cite{vandervaart2008b} and \cite{vandervaart2008}.

\begin{proposition}\label{prop:gp_gen_F}
Consider the prior \eqref{eq:prior_b} on $b$ for a Gaussian process $W^b$
with values in $C([0,1]^d)$ combined with an independent Dirichlet process
prior on $F$. Let $\varepsilon_n^b\rightarrow 0$ satisfy \eqref{eq:gaus_conc_rate}. 
Suppose there exist sequences   $\xi_n \in \H^{b} $ and $\zeta_n^b\rightarrow 0$ such that
\begin{equation}
\label{eq:RKHS_cond_b}
\|\xi_n^b - \xi_{\eta_0}^b\|_\infty \leq \zeta_n^b, \quad \quad \|\xi_n^b\|_{\H^{b}} \leq \sqrt{n} \zeta_n^b, \quad \quad \sqrt{n}\varepsilon_n^b \zeta_n^b \rightarrow 0.
\end{equation}
Suppose further that there exist measurable sets $\mH_n^b$ of functions $\eta^b$
such that 
$\Pi(\eta^b\in (\mH_n^b -t\xi_n^b/\sqrt{n})|X^{(n)})\rightarrow^{P_0} 1$ for every $t\in \R$ and
\eqref{eq:prob_sup2} holds.
Then the posterior distribution satisfies the semiparametric BvM theorem.
\end{proposition}

For the examples of the Riemann-Liouville process and finite Gaussian series prior the preceding proposition
implies the following.

\begin{corollary}\label{cor:RL_F}
Suppose $a_0\in C^\alpha([0,1]^d)$, $b_0 \in C^\beta([0,1]^d)$ and consider the prior \eqref{eq:prior_b} on $b$
with $W^b$ the random series \eqref{eq:series_prior} combined with an independent Dirichlet process prior on $F$. 
If $\alpha,\beta>d/2$ and $d/2<\bar{\beta}<\alpha+\beta-d/2$, then the posterior distribution
satisfies the semiparametric BvM theorem. 
Moreover, when $d=1$ the same result holds with $W^b$ the 
Riemann-Liouville process \eqref{eq:RL_process} with parameter $\bar{\beta}$.
\end{corollary}

For $\alpha,\beta>d/2$, the parameter $\bar\beta$ can always be chosen to satisfy the remaining
condition in the corollary, in which case the BvM theorem holds. The values $\alpha, \beta>d/2$ are one particular pair satisfying
\eqref{eq:double_robust_smoothness}. However, when using product priors, it does not seem possible
to use extra smoothness in one parameter to offset low regularity in the other as in
\eqref{eq:double_robust_smoothness}. To remedy this we consider the propensity score-dependent prior \eqref{eq:prior_dep_EB}.

\begin{corollary}\label{cor:RL_dep_F}
Suppose $a_0\in C^\alpha([0,1]^d)$ and $b_0 \in C^\beta([0,1]^d)$. 
Let $\hat{a}_n$ be an independent estimator satisfying $\|\hat{a}_n\|_\infty = O_{P_0}(1)$ and \eqref{eq:estimator_prop}
for some $\rho_n \rightarrow 0$. Consider the prior \eqref{eq:prior_dep_EB} for $b$, where $W^b$
is the random series \eqref{eq:series_prior} 
combined with  an independent Dirichlet process prior on $F$. 
If $\beta\wedge\bar{\beta} > d/2$ and
$$(n/\log n)^{-(\beta\wedge \bar{\beta})/(2\bar{\beta}+d)} \ll \sigma_n \lesssim 1, \qquad 
\sqrt{n}\rho_n (n/\log n)^{-(\beta \wedge\bar{\beta})/(2\bar{\beta}+d)}\rightarrow 0,$$ 
then the posterior distribution satisfies the semiparametric BvM.
Moreover, when $d=1$ the same result holds with $W^b$ the Riemann-Liouville process \eqref{eq:RL_process} with parameter $\bar{\beta}$.
\end{corollary}

If $\bar{\beta}=\beta$ and $\rho_n = (\log n)^\kappa n^{-{\alpha}/{(2\alpha+1)}}$ is the minimax
rate of estimation, possibly up to a logarithmic factor, then the above conditions reduce to
$\beta>d/2$ and \eqref{eq:double_robust_smoothness}.  If $\beta$ is near the lower limit $d/2$, then
the latter condition requires that $\alpha$ be bigger than nearly $d/2$ as well, but if $\beta$ is large,
then the latter condition will be satisfied for $\alpha$ close to zero.  Thus the estimation method
is able to exploit extra smoothness in $b_0$ to offset lower regularity in $a_0$, in particular if
$0<\alpha\leq d/2$, unlike the standard product Gaussian process priors, where we required both
$\alpha,\beta>d/2$. Since it is still needed that $\beta>d/2$, the preceding corollary does not
give full ``double robustness'' in also taking advantage of extra regularity in $a_0$ if
$0<\beta \leq d/2$. The technical reason is requirement \eqref{eq:prob_sup2}, which is present in all
our theorems, and used in the proofs to establish the LAN expansion of the model. Whether this is a
fundamental limitation of the Bayesian approach or a purely technical artefact is unclear.

If $W^b$ is a mean-zero Gaussian process with covariance kernel $K_{W^b}(z,z')  = \E W_z^b W_{z'}^b$, then the term $W^b + \lambda \hat{a}_n$ in \eqref{eq:prior_dep_EB} is also a mean-zero Gaussian process with data-driven covariance
$$\E[W_z^b + \lambda \hat{a}_n(z)][W_{z'}^b + \lambda \hat{a}_n(z')] = K_{W^b}(z,z') + \sigma_n^2 \hat{a}_n(z) \hat{a}_n(z').$$
In this case, the propensity score-dependent prior corresponds to an easy to implement correction to the prior covariance function. In particular, one can use standard methods for Gaussian process posterior computation, such as Laplace or sparse approximations \cite{rasmussen2006}. In practice, we would also suggest to truncate the estimator $1/\hat{a}_n$ away from 0 for numerical stability. Computational and empirical aspects of this new prior are investigated in the continuous regression model in a companion paper \cite{RaySzabo}, where it is found that incorporating an estimator of the propensity score in this way significantly improves the performance of Gaussian process priors.

\section{Discussion}\label{sec:discussion}
A key technical difficulty for establishing semiparametric BvM results is controlling the ratio \eqref{eq:prior_shift_cond2} (or \eqref{eq:prior_shift_cond1}). While one can use the Cameron-Martin theorem for Gaussian priors, such
results are typically more involved outside the Gaussian setting. The hyper parameter $\lambda$ in
the prior \eqref{eq:prior_dep_EB} removes this obstacle, allowing results for a much wider class of priors. For
instance, one may select $W^b$ in \eqref{eq:prior_dep_EB} to be a truncated prior or sieve prior,
without having to establish \eqref{eq:prior_shift_cond2} directly for those priors.

Such a prior construction generalizes to other models and functionals. Consider a model
$\mathcal{P} = (P_\eta: \eta \in \mH)$ and a parameter $\chi(\eta)$. For a prior of the form
$\eta = W+\lambda \hat{\xi}_n$, where $W$ is a continuous stochastic process,
$\lambda \sim N(0,\sigma_n^2)$ and $\hat{\xi}_n$ is an estimate of the least favourable direction
$\xi_{\eta_0}$ of $\chi$ at $\eta_0$ in the model $\mathcal{P}$, similar results to the above should
hold. We emphasize, however, that such a prior is designed for semiparametric estimation of the
specific functional $\chi$ and will not perform any better for any other functional. It is
thus suitable for estimating a functional of interest in the presence of a high or
infinite-dimensional nuisance parameter that can have a significant impact, as in the model we study
here.

\section{Proofs of the main results}\label{sec:proofs}

\subsection{Proof of Theorem~\ref{thm:general_F}: General prior on $b$ and Dirichlet process prior}

\begin{proof}[Proof of Theorem~\ref{thm:general_F}]
The total variation distance between the posterior distributions based on the prior $\Pi$ and
the prior $\Pi_n (\cdot):= \Pi(\cdot\cap \mH_n)/\Pi(\mH_n)$, which is $\Pi$ conditioned to $\mH_n$, is bounded
above by $2\Pi(\mH_n^c\given X^{(n)})$ (e.g.\ page~142 of \cite{vandervaart1998}). Since this tends to
zero in probability by assumption and the total variation topology is stronger than the weak topology, 
it suffices to show the desired result for the conditioned prior $\Pi_n$ instead of $\Pi$. 

Let $\widehat\chi_n=\chi(\eta_0)+\P_n\widetilde\chi_{\eta_0}$, so that 
it satisfies \eqref{EqEffcieintEstimator} with the remainder term identically zero.
The posterior Laplace transform of the variable $\sqrt n(\chi(\eta)-\widehat\chi_n)$ is given by, for $t\in \R$, 
\begin{align*}
I_n(t) &= \E^{\Pi_n}[e^{t\sqrt{n} (\chi(\eta)-\widehat\chi_n)}\given X^{(n)}]\nonumber\\
&= \int\! \int_{\mH_n}\!\! \frac
{e^{t\sqrt{n}\int (bdF -  b_0 dF_0) -t \G_n [\widetilde{\chi}_{\eta_0}]+\ell_n^{b}(\eta)-\ell_n^{b}(\eta_t)}\, e^{\ell_n^{b}(\eta_t)} }
{\int_{\mH_n} e^{\ell_n^{b}(\eta')} d\Pi(\eta')}  d\Pi(\eta)]\, d\Pi(F|X^{(n)}),
\end{align*}
in view of \eqref{EqDirichletPosterior} and the factorization of the likelihood over $a$ and $b$. 
This is (obviously) true for any $\eta_t$, in particular for the path $\eta_t=\eta_t(\eta)$ defined in
\eqref{eq:LFD_F_no_den}. We shall show that $I_n(t)$ tends in probability to
$\exp (t^2 \| \xi_{\eta_0} \|_{\eta_0}^2/2)$, which is the Laplace transform of a
$N(0, \| \xi_{\eta_0} \|_{\eta_0}^2)$ distribution, for every $t$ in a neighbourhood of 0.  Since convergence of
conditional Laplace transforms in probability implies conditional convergence in distribution in probability (see
Lemma~\ref{LemmaLaplaceTransform} below), this would complete the proof.

At the end of the proof we shall show that, uniformly in $\eta\in\mH_n$,
\begin{align}
\ell_n^b(\eta) - \ell_n^b(\eta_t) 
&= t\G_n[\widetilde{\chi}_{\eta_0}^b] + t\sqrt{n} \int (b_0-b)\, dF_0 +\frac{t^2}{2} \| \xi_{\eta_0}^b \|_{b_0}^2+ o_{P_0}(1),
\label{eq:LAN_remainder_bb}
\end{align}
where $\widetilde\chi_{\eta}^b=B_\eta^ba$ is the component of the efficient influence function in the
$b$ direction (see \eqref{EqEIF}). Inserting this Taylor expansion in the preceding display, we see that 
\begin{align*}
I_n(t)  & = \int\!\int_{\mH_n} \frac
{e^{t\sqrt{n}\int (bdF - b_0 dF_0) +t\sqrt{n} \int (b_0-b)dF_0}\, e^{\ell_n^{b}(\eta_t)}}
{\int_{\mH_n} e^{\ell_n^{b}(\eta')} d\Pi(\eta')}\,   d\Pi(\eta) d\Pi(F\given X^{(n)})\\
&\qquad\qquad\qquad\qquad \times e^{-t \G_n [\widetilde{\chi}_{\eta_0}^f]+\frac{t^2}{2} \| \xi_{\eta_0}^b \|_{b_0}^2 + o_{P_0}(1)},
\end{align*}
where $\widetilde{\chi}_{\eta_0}^f= \widetilde{\chi}_{\eta_0}-\widetilde{\chi}_{\eta_0}^b = b_0 - \chi(\eta_0)$. Note that
the integral in the denominator is a constant relative to $\eta$ and $F$, since all variables are
integrated out. By Fubini's theorem, the double integral without the normalizing constant equals
\begin{align*}
\int_{\mH_n} e^{\ell_n^{b}(\eta_t)} \int e^{t\sqrt{n}\int bd(F-F_0)} \, d\Pi(F|X^{(n)})\, d\Pi(\eta).
\end{align*}
Let $\F_n = n^{-1} \sum_{i=1}^n \delta_{ Z_i}$ denote the empirical distribution of the covariates. 
By assumption \eqref{eq:prob_sup2} we certainly have that
$\sup\{|(\F_n-F_0)b|: b=\Psi(\eta^b), \eta\in\mH_n\}$ tends to zero in probability.
Therefore Lemma~\ref{lem:DP_exp_laplace_trans} below yields that for every $t$ in a neighbourhood
of zero, the preceding display equals 
\begin{align*}
e^{o_{P_0}(1)} \int_{\mH_n} e^{\ell_n^{b}(\eta_t)} e^{t\sqrt{n}\int b d(\F_n-F_0)} e^{\frac{t^2}2\|b-F_0b\|_{L^2(F_0)}^2}\,d\Pi(\eta).
\end{align*}
Since $\|b-b_0\|_{L^2(F_0)} \rightarrow 0$ uniformly on $\mH_n$ 
and $\sqrt{n}\int b d(\F_n-F_0) = \G_n[b_0] + o_{P_0}(1)$ by assumption \eqref{eq:prob_sup2},
the previous display equals
\begin{align*}
e^{t\G_n[b_0] + \frac{t^2}{2}\|b_0-F_0b_0\|_{L^2(F_0)}^2 
+ o_{P_0}(1)}  \int_{\mH_n} e^{\ell_n^{b}(\eta_t)}\,  d\Pi(\eta).
\end{align*}
We insert this in the expression for $I_n(t)$, combine the two exponential terms
using that $\widetilde{\chi}_{\eta_0}^f=b_0 - \chi(\eta_0)$
and $\|b_0-F_0b_0\|_{L^2(F_0)} =\|\xi_{\eta_0}^f\|_{F_0}$, and invoke assumption
\eqref{eq:prior_shift_cond2}, to see that
$I_n(t)$ tends to $e^{t^2\|\xi_{\eta_0}\|_{\eta_0}^2/2}$ in probability. 
The theorem then follows by the convergence of Laplace transforms.

We conclude by a proof of \eqref{eq:LAN_remainder_bb}. This entails an expansion of the log likelihood
$\ell_n^b(\eta)-\ell_n^b(\eta_t)$ along the submodel $\eta_t$. 
We can decompose
\begin{equation}
\begin{split}
\ell_n^b(\eta) - \ell_n^b(\eta_t) 
&= t \G_n [\widetilde{\chi}_{\eta_0}^b] + \sqrt{n}\G_n [\log p_\eta - \log p_{\eta_t} 
- \frac{t}{\sqrt{n}} \widetilde{\chi}_{\eta_0}^b]\\
&\qquad\qquad\qquad + n P_{\eta_0} [\log p_\eta - \log p_{\eta_t}]. \label{eq:lan_least_fav_dir2}
\end{split}
\end{equation}
We shall show that the second term on the right tends to zero in probability, while
the third term tends to the quadratic $t^2\|a_0\|_{b_0}^2/2$, where  $a_0=\xi_{\eta_0}^b$. 

The definition $\eta_u := (\eta^a,\eta_u^b)$ with $\eta_u^b = \eta^b - tu\xi_{\eta_0}^b/\sqrt{n}$, 
for $u \in [0,1]$, gives a path from $\eta_{u=0} = \eta$ (not $\eta_0$!) to $\eta_{u=1}=\eta_t$,
so that $\log p_\eta - \log p_{\eta_t}=g(0)-g(1)$ for $g(u) = \log p_{\eta_u}$. We shall replace 
this difference in both terms on the right of \eqref{eq:lan_least_fav_dir2} 
by the Taylor expansion $g(0)-g(1)=-g'(0)-g''(0)/2-\theta$, where $|\theta|\le \|g'''\|_\infty$. The expansion
will be uniform in $\eta\in\mH_n$, although the dependence of $g$ and $\theta$ on $\eta$ is not indicated 
in the notation.

By explicit calculations the derivatives of $g$ can be seen to be 
\begin{align*}
g'(u) &= -\frac t{\sqrt n} B_{\eta_u}^b a_{0}=-\frac t{\sqrt n} r\bigl(y-\Psi(\eta_u^b)\bigr)a_0,\\
g''(u) &= -\frac{t^2}n r \Psi'(\eta_u^b) a_{0}^2,\qquad
g'''(u)=\frac{t^3}{n^{3/2}} r \Psi''(\eta_u^b) a_{0}^3,\end{align*}
where we have omitted the function arguments $(r,y,z)$.
Since $|\theta|\le \|g'''\|_\infty\lesssim n^{-3/2}$, it follows that both 
$\sqrt n\G_n\theta$ and $n P_0\theta$ tend to zero in probability, uniformly in $\eta\in \mH_n$. 
Since $B_{\eta_0}^b a_{0} = \widetilde{\chi}_{\eta_0}^b$,
\begin{align*} 
g'(0)&=-\frac t{\sqrt n} B_{\eta}^b a_{0}= -\frac t{\sqrt n} \widetilde{\chi}_{\eta_0}^b +\frac t{\sqrt n}r(b-b_0) a_{0},\\
g''(0)&=-\frac{t^2}n r \Psi'(\eta^b) a_{0}^2=-\frac{t^2}n r \Psi'(\eta_0^b) a_{0}^2
-\frac{t^2}n r \bigl(b(1-b)-b_0(1-b_0)\bigr) a_{0}^2
\end{align*}
for $b = \Psi(\eta^b)$, since $\Psi' = \Psi(1-\Psi)$. 

By assumption \eqref{eq:prob_sup2} and Lemma~\ref{LemmaGCnPreservation}, applied with $\mH_{n,1}$ the set of functions
$\sqrt n(b-b_0)$ and $\mH_{n,2}=\{r\}$, 
we have that $\G_n \bigl[r(b-b_0) a_{0}\bigr]\ra 0$ in probability, uniformly in $\{ b=\Psi(\eta^b): \eta\in \mH_n\}$, whence
$\sqrt n \G_n g'(0)=-t\G_n\bigl[ \widetilde{\chi}_{\eta_0}^b ]+o_{P_0}(1)$,
uniformly in $\eta\in \mH_n$. By again assumption \eqref{eq:prob_sup2} and Lemma~\ref{LemmaGCnPreservation},
$\G_n \bigl[r \bigl(b(1-b)-b_0(1-b_0)\bigr) a_{0}^2\bigr]\ra 0$ in probability,
whence $\sqrt n\G_ng''(0)=O_{P_0}(n^{-1/2})\ra 0$ in probability.
We conclude that the second term on the right in \eqref{eq:lan_least_fav_dir2} tends to zero
in probability, uniformly in $\eta\in \mH_n$.

Since $\Psi'(\eta_0^b) = b_0(1-b_0)$ and $ \int b_0(1-b_0) a_0\, dF_0=\|\xi_{\eta_0}^b\|_{b_0}^2$,
\begin{align*}
-nP_{\eta_0} g'(0)&=t\sqrt{n} \int (b_0-b)\,dF_0,\\
-nP_{\eta_0} g''(0)-t^2\| \xi_{\eta_0}^b \|_{b_0}^2 
&= t^2P_{\eta_0}[r\big(b (1-b)- b_0(1-b_0)\bigr) a_0^2]\\
&\lesssim  P_{\eta_0}[r|b-b_0|a_0^2]\le \|b-b_0\|_{L^1(F_0)}\|a_0\|_\infty.
\end{align*}
Therefore  $nP_{\eta_0}[-g'(0)-g''(0)/2]$ is equal to 
$ t\sqrt{n} \int (b_0-b)\,dF_0 + t^2\| \xi_{\eta_0}^b \|_{b_0}^2/2 +o_{P_0}(1)$.
The third term on the right of \eqref{eq:lan_least_fav_dir2}
is equivalent to the same expression.
This concludes the proof of \eqref{eq:LAN_remainder_bb}.
\end{proof}

The preceding proof makes use of the following lemma, which can be considered
a BvM theorem for the Laplace transform of the Dirichlet posterior process.
A proof of the lemma can be found in \cite{RayDP}.

Let $\F_n$ be the empirical distribution of an i.i.d.\ sample $Z_1,\ldots,Z_n$ from a distribution
$F_0$ on a Polish sample space $(\mZ,\mC)$, and given $Z_1,\ldots, Z_n$ let $F_n$ be the distribution of
a draw from the Dirichlet process with base measure $\nu+ n\F_n$. Thus $\nu$ is a finite measure on
$(\mZ,\mC)$, and $F_n\given Z_1,\dots,Z_n \sim DP(\nu + n\F_n)$ is the 
posterior distribution obtained when equipping the distribution of the observations
$Z_1,Z_2,\ldots, Z_n$ with a Dirichlet process prior with base measure $\nu$.
The case $\nu=0$ is allowed.

\begin{lemma}\label{lem:DP_exp_laplace_trans}
Suppose $\mG_n$ are separable classes of measurable functions such that
$\sup_{g\in \mG_n}|\F_ng-F_0g|\rightarrow 0$ in probability and have envelope functions $G_n$ satisfying
$\nu G_n=O(1)$ and $F_0G_n^{2+\delta}=O(1)$ for some $\delta>0$.
Then for every $t$ in a sufficiently small neighbourhood of $0$, in probability,
\begin{align*}
\sup_{g\in \mG_n}\left| \E\bigl[e^{t\sqrt{n}(F_ng-\F_ng)} \given Z_1,\ldots,Z_n\bigr] - e^{t^2 F_0(g-F_0g)^2/2} \right|
 \rightarrow 0.
\end{align*}
\end{lemma}

\subsection{Proof of Theorem~\ref{thm:dep_prior_gen_EB}: propensity score-dependent prior}

\begin{proof}[Proof of Theorem~\ref{thm:dep_prior_gen_EB}]
For the propensity score-dependent prior \eqref{eq:prior_dep_EB} 
the posterior distribution for $\sqrt n(\chi(\eta)-\hat\chi_n)$ is dependent both on 
the data $X^{(n)}$ and the estimator $\hat a_n$, and hence the bounded Lipschitz 
distance between this posterior distribution and the approximating normal distribution
in Definition~\ref{DefBvM} is a function $H(X^{(n)},\hat a_n)$ of this pair of stochastic variables. 
By the assumed stochastic independence of $X^{(n)}$ and $\hat a_n$, the expectation of this 
distance can be disintegrated as
$\E H(X^{(n)},\hat a_n)=\int \E H(X^{(n)},a)\,dP^{\hat a_n}(a)$, where
the expectation inside the integral is relative to $X^{(n)}$ only and
concerns the ``ordinary'' posterior distribution relative to 
the prior \eqref{eq:prior_dep_EB} with $\hat a_n$ set equal to the deterministic function $a$,
i.e.\ the posterior distribution for the prior of the form $\Psi(w+\lambda a)$ on $b$, for a fixed
function $a$ and $(w,\lambda)$ following their prior. Since the bounded Lipschitz distance
is bounded, $\E H(X^{(n)},\hat a_n)$ certainly tends to zero if for every $\eta>0$ there exist sets $\mA_n$ with
$\Pr(\hat a_n\in \mA_n)>1-\eta$ such that $\E H(X^{(n)},a)\rightarrow 0$, uniformly in $a\in \mA_n$.

In view of \eqref{EqPriorConcentration3} there exist sets $A_n$ with $\Pr(\hat a_n\in A_n)\rightarrow 1$
and $\E\Pi \bigl( (w,\lambda): w+(\lambda + tn^{-1/2})a\in \mH_n^b\given X^{(n)}\bigr)\rightarrow1$, uniformly in
$a\in A_n$ (see the lemma below for details). Since we assume that $\|\hat{a}_n\|_\infty = O_{P_0}(1)$ and \eqref{eq:estimator_prop}, 
we can further reduce these sets to $\mA_n=\{a\in A_n: \|a\|_\infty\le M, \|a-a_0\|_{L^2(F_0)} \le M\rho_n\}$, 
and then show that $\E H(X^{(n)},a)\rightarrow 0$ uniformly in $a\in \mA_n$, for (every) fixed $M>0$. 
Thus in the remainder of the proof we fix $\hat a_n$ to be a deterministic sequence $a_n$ in $\mA_n$.

We verify the conditions of Theorem~\ref{thm:general_F}. By \eqref{eq:dep_prior_lambda_cond}--\eqref{eq:prob_sup3},
conditions \eqref{EqPriorConcentration2}--\eqref{eq:prob_sup2}
are met by $\mH_n=\{\eta: \eta^b=w+\lambda a_n, (w,\lambda)\in B_n\}$, for 
\begin{align*}
B_n = \bigl\{ (w,\lambda) :  w +\lambda {a}_n \in \mH_n^b , |\lambda| \leq 2 u_n\sigma_n^2 \sqrt{n}\bigr\}.
\end{align*}
It therefore remains only to control the change of measure \eqref{eq:prior_shift_cond2}.
We need only consider the $b$ part of the integrals, as the $a$ part cancels. 
Because the assumptions become `more true' if $u_n$ is replaced 
by a bigger sequence and $n\sigma_n^2\rightarrow\infty$, 
we may assume that $u_n\rightarrow0$ and $u_nn\sigma_n^2\rightarrow\infty$. 

For the $b$ term, \eqref{eq:prior_shift_cond2} equals
\begin{align}\label{eq:numerator}
\frac{\int_{B_n}  e^{\ell_n^b(w+\lambda {a}_n-ta_0/\sqrt{n})} \phi_{\sigma_n}(\lambda)\,  d\lambda \, d\Pi(w)}
{\int_{B_n}  e^{\ell_n^b(w+\lambda {a_n})} \phi_{\sigma_n}(\lambda)\, d\lambda \, d\Pi(w)},
\end{align}
where $\phi_\sigma$ denotes the probability density function of a $N(0,\sigma^2)$ random
variable. By Lemma~\ref{lem:lik_b_approx}, applied with
$A_n = \{ w+\lambda {a}_n : (w,\lambda) \in B_n\}$, $\xi_n = {a}_n$,
$\xi_0 =  a_0$, $\zeta_n = M\rho_n$, $w_n$ the constant $M$ in the definition of $\mA_n$ 
and $\varepsilon_n = \varepsilon_n^b$, 
\begin{align*}
\sup_{(w,\lambda)\in B_n}\Bigl|\ell_n^b\bigl(w + \lambda {a}_n-\frac t{\sqrt{n}}a_0\bigr)
-\ell_n^b\bigl(w + \bigl(\lambda-\frac t{\sqrt{n}}\bigr){a}_n\bigr)\Bigr| = o_{P_0}(1).
\end{align*}
Furthermore, for $|\lambda|\leq 2u_n\sigma_n^2\sqrt{n}$, we have for the log likelihood ratio of two
normal densities
$$\Bigl|\log \frac{\phi_{\sigma_n}(\lambda)}{\phi_{\sigma_n}(\lambda - t/\sqrt{n})}\Bigr|
\le \frac{|t\lambda|}{\sqrt{n}\sigma_n^2} + \frac{t^2}{2n\sigma_n^2} \rightarrow 0.$$
Consequently, the numerator of \eqref{eq:numerator} equals
\begin{align*}
e^{o_{P_0}(1)} \int_{B_n}  e^{\ell_n^b(w+(\lambda -t/\sqrt{n}) {a}_n)} 
\phi_{\sigma_n}(\lambda-t/\sqrt{n})\,d\lambda\, d\Pi(w).
\end{align*}
By the change of variables $\lambda -t/\sqrt{n} \rightsquigarrow \lambda'$ the ratio
\eqref{eq:numerator} therefore equals, for
$B_{n,t}=\{ (w,\lambda): (w,\lambda+t/\sqrt{n})\in B_n\}$,
\begin{align*}
e^{o_{P_0}(1)} \frac{\int_{B_{n,t}}  e^{\ell_n^b(w+\lambda' {a}_n)} \phi_{\sigma_n}(\lambda')\, d\lambda'\, d\Pi(w)}
{\int_{B_n}  e^{\ell_n^b(w+\lambda {a}_n)} \phi_{\sigma_n}(\lambda)\, d\lambda \, d\Pi(w)} 
= e^{o_{P_0}(1)}\frac{\Pi(B_{n,t}|X^{(n)})}{\Pi(B_n|X^{(n)})}. 
\end{align*}
Since $\Pi(B_n|X^{(n)})=1-o_{P_0}(1)$, it remains to show that $\Pi(B_{n,t}|X^{(n)})=1-o_{P_0}(1)$. 

The set $B_{n,t}$ is the intersection of the sets in assumptions 
\eqref{EqPriorConcentration3} (with $\hat a_n=a_n$) and \eqref{eq:dep_prior_lambda_cond}, except that
the restriction on $\lambda$ in $B_{n,t}$ is $|\lambda+t/\sqrt{n}| \le 2u_n \sqrt{n}\sigma_n^2$,
whereas in \eqref{eq:dep_prior_lambda_cond} the restriction is $|\lambda| \le u_n \sqrt{n}\sigma_n^2$.
Since $t/\sqrt n\ll u_n \sqrt{n}\sigma_n^2$ by construction,
the latter restriction implies the former, and hence $\Pi(B_{n,t}|X^{(n)})=1-o_{P_0}(1)$ by assumption.
\end{proof}

\begin{lemma}
For given $v$ define $A_n(v)$ to be the set of all $a$ such that 
$\E\Pi \bigl( (w,\lambda): w+(\lambda + tn^{-1/2})a\in \mH_n^b\given X^{(n)}\bigr)>1-v$.
If \eqref{EqPriorConcentration3} holds, then there exists $v_n\downarrow 0$ such that $\Pr(\hat a_n\in A_n(v_n))\rightarrow 1$. 
\end{lemma}

\begin{proof}
For given $a$ and $x$, define $$G_n(a,x)=\Pi \bigl( (w,\lambda): w+(\lambda + tn^{-1/2})a\in \mH_n^b\given X^{(n)}=x\bigr).$$
Then the given expectation is $H_n(a):=\E G_n(a, X^{(n)})$ and $A_n(v)=\{a: H_n(a)>1-v\}$.
By \eqref{EqPriorConcentration3}, the dominated convergence theorem and the independence of $\hat a_n$ and $X^{(n)}$,
we have  $\E H_n(\hat a_n)=\E G_n(\hat a_n,X^{(n)})\ra 1$.
Since  $0\le H_n(a)\le 1$, this implies that $H_n(\hat a_n)\rightarrow^{P}1$.
Then $\Pr(H_n(\hat a_n)>1-v_n)\ra1$, for $v_n\downarrow 0$ sufficiently slowly by a standard argument.
\end{proof}

\subsection{Proofs for Section~\ref{SectionGaussianPriors}: Gaussian process priors}

\begin{proof}[Proof of Proposition~\ref{prop:gp_gen_F}]
We verify the conditions of Theorem~\ref{thm:general_F}. By Lemma~\ref{lem:contrac_prod_b} 
with norm $\|\cdot\|_\infty$, the posterior distribution of $b$ contracts about $b_0$ at rate $\varepsilon_n^b$ in
$L^2(F_0)$. For $\mH_n^b$ the sets as in the statement of the proposition, define
$$\mH_n = \bigl\{(\eta^a,\eta^b): \eta^b\in\mH_n^b, \|\Psi(\eta^b)-b_0\|_{L^2} \leq \varepsilon_n^b\bigr\}.$$
Then $\Pi(\mH_n\given X^{(n)}) \rightarrow^{P_0} 1$ as $n\rightarrow \infty$, by assumption.
It follows that $\mH_n$ satisfies conditions \eqref{EqPriorConcentration2}--\eqref{EqConsistencyb2},
while \eqref{eq:prob_sup2} is satisfied by assumption.

It remains to verify \eqref{eq:prior_shift_cond2}.
Following \cite{castillo2012}, we first approximate the perturbation $\eta_t^b$ by an element in the
RKHS $\H^b$ and then apply the Cameron-Martin theorem. Let $\xi_n^b\in \H^b$ satisfy
\eqref{eq:RKHS_cond_b}, and set $\eta_{n,t}= \eta_{n,t}(\eta^b)= \eta^b - t\xi_n^b/\sqrt{n}$. 
By the Cameron-Martin theorem (see Lemma~\ref{LemCM}), the distribution $\Pi_{n,t}$
of $\eta_{n,t}$ if $\eta^b$ is distributed according to the prior $\Pi$ has Radon-Nikodym density
$$\frac{d\Pi_{n,t}}{d\Pi}(\eta^b)=e^{tU_n(\eta^b)/\sqrt{n} - t^2 \|\xi_n^b\|_{\H^b}^2/(2n)},$$
where $U_n(\eta^b)$ is a centered Gaussian variable with variance $\|\xi_n^b\|_{\H^b}^2$
if $\eta^b\sim \Pi$, and $\|\cdot\|_{\H^b}$ is the RKHS norm of the Gaussian process $\eta^b$.
By the univariate Gaussian tail bound,
\begin{equation}
\label{EqUnivariateGaussian}
\Pi\bigl( \eta^b: |U_n(\eta^b)| > M\sqrt{n} \varepsilon_n^b \| \xi_n^b\|_{\H^b}\bigr) \leq
2e^{-M^2n(\varepsilon_n^b)^2/2}.
\end{equation}
Consequently, by Lemma~\ref{lem:small_prior_prob} the posterior measure of 
the set in the display tends to 0 in probability, for large enough $M$. Hence the sets
\begin{align*}
B_n = \bigl\{\eta^b: |U_n({\eta^b})| \leq M\sqrt{n} \varepsilon_n^b \| \xi_n^b\|_{\H^b}\bigr\}\cap \mH_n^b
\end{align*}
also satisfy $\Pi(B_n|X^{(n)}) \ra 1$ in probability.  
On the sets $B_n$, in view of \eqref{eq:RKHS_cond_b},
\begin{align}\label{eq:CM}
\Bigl|\log \frac{d\Pi_{n,t}}{d\Pi}(\eta^b)\Bigr|\le { M|t|\sqrt{n}\varepsilon_n^b \zeta_n^b + \frac{t^2}{2}(\zeta_n^b)^2}\ra 0.
\end{align}
Furthermore, by Lemma~\ref{lem:lik_b_approx} applied with $A_n=B_n$, $\xi_0 = \xi_{\eta_0}^b$,
$\varepsilon_n=\varepsilon_n^b$, $\zeta_n = \zeta_n^b$ and $w_n$ a sufficiently large fixed
constant, we have 
$$\sup_{\eta^b\in B_n}|\ell_n^b(\eta_{n,t})-\ell_n^b(\eta_t^b)| = o_{P_0}(1).$$ 
(Note that condition \eqref{eq:sup_lem_lik_b_approx} holds by assumption \eqref{eq:prob_sup3}
and Lemma~\ref{lem:prob_to_expect}.) By the last display followed by the change of integration variable
$\eta^b-t\xi_n^b/\sqrt n\rightsquigarrow v$,
\begin{align*}
\frac{\int_{B_n} e^{\ell_n^b(\eta_{t}^b)} d\Pi(\eta^b)}{\int_{B_n} e^{\ell_n^b(\eta^b)} d\Pi(\eta^b)} 
=\frac{\int_{B_n} e^{\ell_n^b(\eta_{n,t})} d\Pi(\eta^b)}{\int_{B_n} e^{\ell_n^b(\eta^b)} d\Pi(\eta^b)} \,e^{o_{P_0}(1)}
=\frac{\int_{B_{n,t}} e^{\ell_n^b (v)} \, d\Pi_{n,t}(v) }{\int_{B_n} e^{\ell_n^b(\eta^b)} d\Pi(\eta^b)}\,e^{o_{P_0}(1)},
\end{align*}
where $B_{n,t} = B_n-t\xi_n^b/\sqrt{n}$. By \eqref{eq:CM} we can next replace $\Pi_{n,t}$ 
in the numerator by $\Pi$ at the
cost of another multiplicative $1+o_{P_0}(1)$ term. This turns the quotient into the ratio
$\Pi(B_{n,t}|X^{(n)})/\Pi(B_n|X^{(n)})$. We have already shown that $\Pi(B_n|X^{(n)})=1-o_{P_0}(1)$,
so it suffices to show the same result holds true for the numerator. Now
\begin{align*}
B_{n,t}^c  &= \bigl\{ v: v+{t\xi_n^b}/{\sqrt{n}} \not\in \mH_n^b\bigr\} 
\cup \bigl\{ v: \|\Psi(v+{t\xi_n^b}/{\sqrt{n}} )-b_0\|_{L^2(F_0)} > \varepsilon_n^b \bigr\}  \\
&\qquad\qquad\cup \bigl\{ v: |U_n(v+{t\xi_n^b}/{\sqrt{n}}  )| > M\sqrt{n} \varepsilon_n^b \| \xi_n^b\|_{\H^b} \bigr\}.
\end{align*}
The posterior probability of the first set tends to zero in probability by assumption.  Since
$\bigl\|\Psi(\eta^b + t\xi_n^b/\sqrt{n})-\Psi(\eta^b)\bigr\|_{L^2(F_0)}\lesssim
\|\xi_n^b/\sqrt n\|_{L^2(F_0)} \lesssim 1/\sqrt n$,
the second set is contained in
$\{ \eta^b: \|\Psi(\eta^b)-b_0\|_{L^2(F_0)} > \varepsilon_n^b -C/\sqrt{n}\}$, which has posterior
probability $o_{P_0}(1)$ by Lemma~\ref{lem:contrac_prod_b}, possibly after replacing
$\varepsilon_n^b$ by a multiple of itself. For the third set, we use that 
$U_n(\eta^b+t\xi_n^b/\sqrt{n})\sim N(-t\|\xi_n^b\|_\H^2/\sqrt n, \|\xi_n^b\|_\H^2)$ if $\eta^b$ is distributed
according to the prior, by Lemma~\ref{LemCM}. Since the mean $t\|\xi_n^b\|_\H^2/\sqrt n$ of this variable
is negligible relative to its standard deviation, 
$\Pi \bigl( |U_n(\eta^b+t\xi_n^b/\sqrt{n}) | > M\sqrt{n} \varepsilon_n^b \| \xi_n^b\|_{\H^b}\bigr)$
differs not substantially from the left side of \eqref{EqUnivariateGaussian}, whence
it is also exponentially small, so that again Lemma~\ref{lem:small_prior_prob} applies
to see that the posterior probability tends to zero.
\end{proof}

\begin{proof}[Proof of Corollary~\ref{cor:RL_F}]
The proof follows by verifying the conditions of Proposition~\ref{prop:gp_gen_F}, separately for
the two prior processes.

\textbf{Series prior} \eqref{eq:series_prior}: Using the form of the concentration
function in the proof of Theorem~4.5 of \cite{vandervaart2008}, we see that \eqref{eq:gaus_conc_rate} is satisfied for
$$\varepsilon_n^b = n^{-\frac{\beta\wedge \bar{\beta}}{2\bar{\beta}+d}} \log n.$$
Condition~\eqref{eq:prob_sup2} is verified in Lemma~\ref{lem:sup_series},
under the assumption $\beta \wedge \bar{\beta} >d/2$.

It thus remains only to establish \eqref{eq:RKHS_cond_b}, the approximation by elements of the RKHS. Write $J = J_{\bar{\beta}}$ and define $V_J = \text{span}(\psi_{jk}: j\leq J, k)$. Recall that the RKHS of the Gaussian series prior \eqref{eq:series_prior} equals
\begin{equation}\label{eq:RKHS_series}
\H^b = \Bigl\{ w\in V_J: \|w\|_{\H^b}^2 := \sum_{j\leq J} \sum_k \sigma_j^{-2} |\langle w,\psi_{jk}\rangle_{L^2}|^2 < \infty \Bigr\}.
\end{equation}
From the computations in Theorem 4.5 of \cite{vandervaart2008}, one gets that for $\xi_{\eta_0}^b = a_0 \in C^\alpha$ and any $\zeta_n^b \gtrsim n^{-\alpha/(2\bar{\beta}+d)}$,
\begin{align}\label{RKHS_approx_b}
\inf_{\xi: \|\xi-a_0\|_\infty \leq \zeta_n^b} \|\xi\|_{\H^b} \lesssim \begin{cases} (\zeta_n^b)^{-\frac{r-\alpha+d/2}{\alpha} \wedge 0} & \quad \text{if } r-\alpha+d/2 \neq 0,\\
\log (1/\zeta_n^b) & \quad \text{if } r-\alpha+d/2=0.
\end{cases}
\end{align}
If follows that \eqref{eq:RKHS_cond_b} is satisfied if we can choose $\zeta_n^b \to 0$ so that the right side of the display is bounded above by $\sqrt n\zeta_n^b$ and $\sqrt n\varepsilon_n^b\zeta_n^b\rightarrow 0$.
\begin{itemize}
\item If $r-\alpha+d/2>0$, then \eqref{RKHS_approx_b} is bounded by $\sqrt{n}\zeta_n^b$ for $\zeta_n^b \gtrsim n^{-\alpha/(2r+d)}$. Since we also require $\zeta_n^b \gtrsim n^{-\alpha/(2\bar{\beta}+d)}$, we may take $\zeta_n^b \sim n^{-\alpha/(2\bar{\beta}+d)} \vee n^{-\alpha/(2r+d)} = n^{-\alpha/(2\bar{\beta}+d)}$ since $r\leq \beta \wedge \bar{\beta}$ by assumption. Then $\sqrt{n} \varepsilon_n^b \zeta_n^b \rightarrow 0$ for $\beta \wedge \bar{ \beta}>d/2 + \bar{\beta}-\alpha$.
\item If $r-\alpha+d/2<0$, then \eqref{RKHS_approx_b} is bounded by $\sqrt{n}\zeta_n^b$ and also $\zeta_n^b \gtrsim n^{-\alpha/(2\bar{\beta}+d)}$ for the choice $\zeta_n^b \sim n^{-1/2} \vee n^{-\alpha/(2\bar{\beta}+d)}$. If $1/2 \leq \alpha/(2\bar{\beta}+d)$, then $\sqrt{n}\eps_n^b \zeta_n^b \sim \eps_n^b \to 0$. If $1/2 > \alpha/(2\bar{\beta}+d)$, then $\sqrt{n}\eps_n^b \zeta_n^b \sim n^{(\bar{\beta}+d/2-\beta\wedge\bar{\beta}-\alpha)/(2\bar{\beta}+d)} (\log n)\to 0$ for $\beta \wedge \bar{ \beta}>d/2 + \bar{\beta}-\alpha$.
\item If $r-\alpha+d/2=0$, then one takes $\zeta_n \sim [(\log n)^{1/2} n^{-1/2}]\vee n^{-\alpha/(2\bar{\beta}+d)}$. This is the same as the previous case apart from the extra logarithmic factor, so $\sqrt{n}\eps_n^b \zeta_n^b \to 0$ under exactly the same conditions.
\end{itemize}
Examining all the cases, the above can be summarized as \eqref{eq:RKHS_cond_b} holds if $\beta\wedge\bar{\beta}>[\bar{\beta}-\alpha+d/2]\vee 0$. Together with the condition $\beta \wedge \bar{\beta}>d/2$ needed to verify \eqref{eq:prob_sup2} above, this is equivalent to $\alpha,\beta>d/2$ and $d/2 < \bar{\beta}<\alpha+\beta-d/2$.

\textbf{Riemann-Liouville prior} \eqref{eq:RL_process}: 
The proof follows in much the same way. Using the form of the concentration
function in Theorem~4 of Castillo \cite{castillo2008}, we see that \eqref{eq:gaus_conc_rate} is satisfied for
$$\varepsilon_n^b = n^{-\frac{\beta\wedge \bar{\beta}}{2\bar{\beta}+1}} (\log n)^\kappa,$$
where  $\kappa$ is function of $(\beta,\bar\beta)$, given explicitly in \cite{castillo2008}.
Condition~\eqref{eq:prob_sup2} is verified in Lemma~\ref{lem:sup_RL},
under the assumption $\beta \wedge \bar{\beta} >1/2$. 

It thus remains to establish \eqref{eq:RKHS_cond_b}. 
Recall that the RKHS of the Riemann-Liouville process is the Sobolev space $H^{\bar{\beta}+1/2}$.
From the computations in Theorem~4 of \cite{castillo2008}, one
gets that for $\xi_{\eta_0}^b = a_0 \in C^\alpha$, as $\zeta_n^b\rightarrow 0$,
\begin{align*}
\inf_{\xi: \|\xi-a_0\|_\infty \leq \zeta_n^b} \|\xi\|_{\H^b} \lesssim (\zeta_n^b)^{-\frac{\bar{\beta}-\alpha+1/2}{\alpha} \wedge 0}.
\end{align*}
If follows that \eqref{eq:RKHS_cond_b} is satisfied if we can choose $\zeta_n^b$ so that the right
side of the display is bounded above by $\sqrt n\zeta_n^b$ and
$\sqrt n\varepsilon_n^b\zeta_n^b\rightarrow 0$.  If $\bar{\beta}\leq\alpha
-1/2$, 
simply set $\xi_n^b = \xi_{\eta_0}^b$ and $\zeta_n^b = n^{-1/2} \|\xi_{\eta_0}^b\|_{\H^b}$.  If
$\bar{\beta}>\alpha -1/2$, take $\zeta_n^b = n^{-\frac{\alpha}{2\bar{\beta}+1}}$, so that
$\sqrt{n} \varepsilon_n^b \zeta_n^b \rightarrow 0$ for
$\beta \wedge \bar{ \beta}>1/2 + \bar{\beta}-\alpha$. 
A careful analysis of all cases shows that these inequalities, together with
the requirement $\beta \wedge \bar{\beta}>1/2$, are equivalent to $\alpha,\beta>1/2$ and
$1/2 < \bar{\beta}<\alpha+\beta-1/2$.
\end{proof}

\begin{proof}[Proof of Corollary~\ref{cor:RL_dep_F}]
We verify the conditions of Theorem~\ref{thm:dep_prior_gen_EB}, 
where we replace $\hat a_n$ by a deterministic sequence
with $\|a_n\|_\infty=O(1)$ as explained in the proof of Theorem~\ref{thm:dep_prior_gen_EB}.  
Since $\varepsilon_n^b = n^{-(\beta\wedge \bar{\beta})/(2\bar{\beta}+d)}(\log n)^\kappa$ solves \eqref{eq:gaus_conc_rate} (see proof of Corollary \ref{cor:RL_F}), the contraction rate follows from Lemma~\ref{lem:contrac_dep_EB}.
Together with Lemmas~\ref{lem:sup_RL_dep} and \ref{lem:sup_series_dep} for the Riemann-Liouville and series priors, respectively, this verifies conditions \eqref{EqPriorConcentration3}--\eqref{eq:prob_sup3}.
To verify \eqref{eq:dep_prior_lambda_cond}, we use the Gaussian tail
inequality to see that $\Pi(|\lambda| \geq u_n\sigma_n \sqrt{n})\leq 2e^{-u_n^2 n\sigma_n^2/2}$.
This is bounded above by  $e^{-Ln(\varepsilon_n^b)^2}$ for $u_n\rightarrow0$ sufficiently slowly, 
since  $\varepsilon_n^b = o(\sigma_n)$ by assumption. 
Lemma~\ref{lem:small_prior_prob} now implies \eqref{eq:dep_prior_lambda_cond}.
\end{proof}

\noindent \textbf{Acknowledgements:} We would like to thank two referees for helpful comments and for drawing several references to our attention. The first author would also like to thank Richard Nickl for helpful conversations on symmetrization. Much of this work was done while Kolyan Ray was a postdoc at Leiden University.


\bibliography{double_robust}{}
\bibliographystyle{acm}

\newpage

In the next sections, we present an additional theorem, putting a general prior on $(a,b,f)$, and provide the remaining proofs.


\section{General prior on $a$, $b$ and $f$}
\label{sec:general_prior_f}

Both Theorem~\ref{thm:general_F} and Theorem~\ref{thm:dep_prior_gen_EB} put a Dirichlet process prior on $F$.
In this section we study putting a prior on a density of $F$. In the main theorem, we consider a general prior $\Pi$ on the triple $(a,b,f)$, 
or equivalently on the triple $\eta=(\eta^a, \eta^b, \eta^f)$ constructed through the parame\-trization \eqref{eq:paramtrization}.
This leads to an analogue of Theorem~\ref{thm:general_F}. A similar analogue of 
Theorem~\ref{thm:dep_prior_gen_EB} is also possible, but omitted. Later in the section we specialize the main theorem to
Gaussian process priors.

Define $\eta_t(\eta) =\eta_t(\eta; n,\xi_{\eta_0})$ to be a perturbation of $\eta=(\eta^a,\eta^b,\eta^f)$ 
in the least favourable direction as follows:
\begin{equation}\label{eq:LFD_f_den}
\eta_t(\eta) 
= \Bigl(\eta^a, \eta^b - \frac{t}{\sqrt{n}}\xi_{\eta_0}^b , 
\eta^f - \frac{t}{\sqrt{n}}\xi_{\eta_0}^f  -\log \textstyle\int e^{\eta^f -t\xi_{\eta_0}^f/\sqrt{n}}\,dz\Bigr).
\end{equation}

\begin{theorem}\label{thm:general_f}
Consider an arbitrary prior $\Pi$ on $\eta = (\eta^a,\eta^b,\eta^f)$. 
Assume that there exist measurable sets $\mH_n$ of functions satisfying
\begin{align}
 \Pi(\eta\in\mH_n|X^{(n)}) &\rightarrow^{P_0} 1,\label{EqPriorConcentration}\\
 \sup_{b=\Psi(\eta^b):\eta\in \mH_n} \|b-b_0\|_{L^2(F_0)} &\rightarrow 0,\label{EqConsistencyb}\\
 \sup_{f=e^{\eta^f}\!/\!\int e^{\eta^f}dz: \eta \in \mH_n} \|f-f_0\|_1  &\rightarrow 0,\label{EqConsistencyf}\\
\sup_{b=\Psi(\eta^b): \eta \in \mH_n} \bigl|\G_n[b-b_0]\bigr| &\rightarrow^{P_0} 0, \label{eq:prob_sup}
\end{align}
and also 
\begin{align}\label{eq:no_bias_f}
&\sup_{b=\Psi(\eta^b), f=e^{\eta^f}\!/\!\int e^{\eta^f}dz: \eta\in \mH_n} \Bigl| \sqrt{n}\int (b-b_0)(f-f_0)\,dz \Bigr| \rightarrow 0.
\end{align}
If for the path $\eta_t(\eta)$ given in \eqref{eq:LFD_f_den} and every $t$,
\begin{align}\label{eq:prior_shift_cond1}
\frac{\int_{\mH_n} \prod_{i=1}^np_{\eta_t(\eta)}(X_i)\, d\Pi(\eta)}{\int_{\mH_n} \prod_{i=1}^np_{\eta}(X_i) \, d\Pi(\eta)} \rightarrow^{P_0} 1,
\end{align}
then the posterior distribution of $\chi(\eta)$ satisfies the BvM theorem. 
\end{theorem}

\begin{proof}
The total variation distance between the posterior distributions based on the prior $\Pi$ and
the prior $\Pi_n (\cdot):= \Pi(\cdot\cap \mH_n)/\Pi(\mH_n)$, which is $\Pi$ conditioned to $\mH_n$, is bounded
above by $2\Pi(\mH_n^c|X^{(n)})$ (e.g.\ page~142 of \cite{vandervaart1998}). Since this tends to
zero in probability by assumption and the total variation topology is stronger than the weak topology, 
it suffices to show the desired result for the conditioned prior $\Pi_n$ instead of $\Pi$. 


Let $\widehat\chi_n=\chi(\eta_0)+\P_n\widetilde\chi_{\eta_0}$, so that 
it satisfies \eqref{EqEffcieintEstimator} with the remainder term identically zero.
The posterior Laplace transform of the variable $\sqrt n(\chi(\eta)-\widehat\chi_n)$ is given by, for $t\in \R$, 
\begin{align}
I_n(t) &= \E^{\Pi_n}[e^{t\sqrt{n} (\chi(\eta)-\widehat\chi_n)}|X^{(n)}]\nonumber\\
&= \frac{\int_{\mH_n} e^{t\sqrt{n} \int (bf -  b_0f_0)dz -t\G_n [\widetilde{\chi}_{\eta_0}]+\ell_n(\eta)-\ell_n(\eta_t)}
\, e^{\ell_n(\eta_t)}\, d\Pi(\eta)}
{\int_{\mH_n} e^{\ell_n(\eta)} d\Pi(\eta)},\label{eq:laplace_transform}
\end{align}
for any $\eta_t$, in particular for the path $\eta_t=\eta_t(\eta)$ defined in
\eqref{eq:LFD_f_den}. We shall show that $I_n(t)$ tends in probability to
$\exp (t^2 \| \xi_{\eta_0} \|_{\eta_0}^2/2)$, which is the Laplace transform of a
$N(0, \| \xi_{\eta_0} \|_{\eta_0}^2)$ distribution, for every $t\in\RR$.  Since convergence of
Laplace transforms in probability implies convergence in distribution in probability (see
Lemma~\ref{LemmaLaplaceTransform}), this would complete the proof.

In view of assumption \eqref{eq:prior_shift_cond1} it certainly suffices to show that the exponent
of the first exponential in the numerator of \eqref{eq:laplace_transform} tends to $t^2 \| \xi_{\eta_0} \|_{\eta_0}^2/2$ in
probability, uniformly in $\eta\in\mH_n$.  This entails an expansion of the likelihood
$\ell_n(\eta)-\ell_n(\eta_t)$ along the submodel $\eta_t$. This submodel consists of perturbations
in the directions of $b$ and $f$. Since the likelihood factorizes in these parameters, whence the log
likelihood is additive, the expansion can be performed separately in the perturbations in
the two parameters and the results added. In a slight abuse of notation, we write 
$\eta_t^b = (\eta^a,\eta_t^b,\eta^f)$ and $\eta_t^f=(\eta^a,\eta^b,\eta_t^f)$ for
the path \eqref{eq:LFD_f_den} with the perturbations with $f$ and $b$ held
fixed, respectively, and leave off the argument $\eta$ of $\eta_t=\eta_t(\eta)$.  
For $\widetilde\chi_{\eta}^b=B_\eta^ba$ and 
$\widetilde\chi_{\eta}^f=B_\eta^f\bigl(b-\chi(\eta)\bigr)= b-\chi(\eta)$ 
the components of the efficient influence function in the
$b$ and $f$ directions, respectively (see \eqref{EqEIF}), we shall show that, uniformly in $\eta\in\mH_n$,
\begin{align}
\ell_n^b(\eta) - \ell_n^b(\eta_t^b) 
&= t\G_n[\widetilde{\chi}_{\eta_0}^b] + t\sqrt{n} \int (b_0-b)f_0\, dz +\frac{t^2}{2} \| \xi_{\eta_0}^b \|_{b_0}^2+ o_{P_0}(1),
\hspace{-2em}
\label{eq:LAN_remainder_b}\\
\ell_n^f(\eta) - \ell_n^f(\eta_t^f) 
&= t\G_n[\widetilde{\chi}_{\eta_0}^f] + t\sqrt{n} \int b_0(f_0-f)\, dz + \frac{t^2}{2} \| \xi_{\eta_0}^f \|_{F_0}^2 + o_{P_0}(1). 
\hspace{-2em}
\label{eq:LAN_remainder_f}
\end{align}
Adding these results yields
\begin{align*}
&t\sqrt{n} \int (bf - b_0f_0)dz -t\G_n [\widetilde{\chi}_{\eta_0}]+\ell_n(\eta)-\ell_n(\eta_t)\\
&\qquad\qquad\qquad =  t\sqrt{n} \int (b-b_0)(f-f_0) dz + \frac{t^2}{2} \| \xi_{\eta_0} \|_{\eta_0}^2 + o_{P_0}(1).
\end{align*}
By assumption \eqref{eq:no_bias_f} the first term on the right side tends to zero uniformly over $\mH_n$.
The left side is the exponent in the right of \eqref{eq:laplace_transform}
and the proof is complete. We finish by proving \eqref{eq:LAN_remainder_b} and \eqref{eq:LAN_remainder_f}.

\textbf{$b$ term} \eqref{eq:LAN_remainder_b}: 
We can decompose
\begin{align}
\ell_n^b(\eta) - \ell_n^b(\eta_t^b) 
&= t \G_n [\widetilde{\chi}_{\eta_0}^b] + \sqrt{n}\G_n [\log p_\eta - \log p_{\eta_t^b} 
- \frac{t}{\sqrt{n}} \widetilde{\chi}_{\eta_0}^b]\nonumber\\
&\qquad\qquad\qquad + n P_{\eta_0} [\log p_\eta - \log p_{\eta_t^b}]. \label{eq:lan_least_fav_dir}
\end{align}
We shall show that the second term on the right tends to zero in probability, while
the third term tends to the quadratic $t^2\|a_0\|_{b_0}^2/2$, where  $a_0=\xi_{\eta_0}^b$. 

The definition $\eta_u := (\eta^a,\eta_u^b,\eta^f)$ with $\eta_u^b = \eta^b - tu\xi_{\eta_0}^b/\sqrt{n}$, 
for $u \in [0,1]$, gives a path from $\eta_{u=0} = \eta$ (not $\eta_0$!) to $\eta_{u=1}=\eta_t^b$,
so that $\log p_\eta - \log p_{\eta_t^b}=g(0)-g(1)$, for $g(u) = \log p_{\eta_u}$. We shall replace 
this difference in both terms on the right of \eqref{eq:lan_least_fav_dir} 
by the Taylor expansion $g(0)-g(1)=-g'(0)-g''(0)/2-\theta$, where $|\theta|\le \|g'''\|_\infty$. The expansion
will be uniform in $\eta\in\mH_n$, although the dependence of $g$ and $\theta$ on $\eta$ is not indicated 
in the notation.

By explicit calculations the derivatives of $g$ can be seen to be 
\begin{align*}
g'(u) &= -\frac t{\sqrt n} B_{\eta_u}^b a_{0}=-\frac t{\sqrt n} r\bigl(y-\Psi(\eta_u^b)\bigr)a_0,\\
g''(u) &= -\frac{t^2}n r \Psi'(\eta_u^b) a_{0}^2,\qquad
g'''(u)=\frac{t^3}{n^{3/2}} r \Psi''(\eta_u^b) a_{0}^3,\end{align*}
where we have omitted the function arguments $(r,y,z)$.
Since $|\theta|\le \|g'''\|_\infty\lesssim n^{-3/2}$, it follows that both 
$\sqrt n\G_n\theta$ and $n P_0\theta$ tend to zero in probability, uniformly in $\eta\in \mH_n$. 
Since $B_{\eta_0}^b a_{0} = \widetilde{\chi}_{\eta_0}^b$,
\begin{align*} 
g'(0)&=-\frac t{\sqrt n} B_{\eta}^b a_{0}= -\frac t{\sqrt n} \widetilde{\chi}_{\eta_0}^b +\frac t{\sqrt n}r(b-b_0) a_{0},\\
g''(0)&=-\frac{t^2}n r \Psi'(\eta^b) a_{0}^2=-\frac{t^2}n r \Psi'(\eta_0^b) a_{0}^2
-\frac{t^2}n r \bigl(b(1-b)-b_0(1-b_0)\bigr) a_{0}^2
\end{align*}
for $b = \Psi(\eta^b)$, since $\Psi' = \Psi(1-\Psi)$. 

By assumption \eqref{eq:prob_sup} and Lemma~\ref{LemmaGCnPreservation}, applied with $\mH_{n,1}$ the set of functions
$\sqrt n(b-b_0)$ and $\mH_{n,2}=\{r\}$, 
we have that $\G_n \bigl[r(b-b_0) a_{0}\bigr]\ra 0$ in probability, uniformly in $\{ b=\Psi(\eta^b): \eta\in \mH_n\}$, whence
$\sqrt n \G_n g'(0)=-t\G_n\bigl[ \widetilde{\chi}_{\eta_0}^b ]+o_{P_0}(1)$,
uniformly in $\eta\in \mH_n$. By again assumption \eqref{eq:prob_sup} and Lemma~\ref{LemmaGCnPreservation},
$\G_n \bigl[r \bigl(b(1-b)-b_0(1-b_0)\bigr) a_{0}^2\bigr]\ra 0$ in probability,
whence $\sqrt n\G_ng''(0)=O_{P_0}(n^{-1/2})\ra 0$ in probability.
We conclude that the second term on the right in \eqref{eq:lan_least_fav_dir} tends to zero
in probability, uniformly in $\eta\in \mH_n$.

Since $\Psi'(\eta_0^b) = b_0(1-b_0)$ and $ \int b_0(1-b_0) a_0\, dF_0=\|\xi_{\eta_0}^b\|_{b_0}^2$,
\begin{align*}
-nP_{\eta_0} g'(0)&=t\sqrt{n} \int (b_0-b)\,dF_0,\\
-nP_{\eta_0} g''(0)-t^2\| \xi_{\eta_0}^b \|_{b_0}^2 
&= t^2P_{\eta_0}[r\big(b (1-b)- b_0(1-b_0)\bigr) a_0^2]\\
&\lesssim  P_{\eta_0}[r|b-b_0|a_0^2]\le \|b-b_0\|_{L^1(F_0)}\|a_0\|_\infty.
\end{align*}
Therefore  $nP_{\eta_0}[-g'(0)-g''(0)/2]$ is equal to 
$ t\sqrt{n} \int (b_0-b)\,dF_0 + t^2\| \xi_{\eta_0}^b \|_{b_0}^2/2 +o_{P_0}(1)$.
The third term on the right of \eqref{eq:lan_least_fav_dir}
is equivalent to the same expression.
This concludes the proof of \eqref{eq:LAN_remainder_b}.

\textbf{$f$ term} \eqref{eq:LAN_remainder_f}:
We use the same decomposition \eqref{eq:lan_least_fav_dir}, but with $b$ replaced by $f$. 
Define the path $\eta_u := (\eta^a,\eta^b,\eta_u^f)$ with $\eta_u^f = \eta^f - tu\xi_{\eta_0}^f/\sqrt{n} +\log c_u$,
for $c_u^{-1}=\int e^{\eta^f - tu\xi_{\eta_0}^f/\sqrt{n}}\,dz$ the norming constant and $u\in[0,1]$. Then $f_u=e^{\eta_u^f}$
is a one-dimensional exponential family in $u$ with score function 
$\dot f_u/f_u=-(t/\sqrt n)(\xi_{\eta_0}^f-F_u\xi_{\eta_0}^f)$ (note that $\dot c_u/c_u=(t/\sqrt n)\int \xi_{\eta_0}^ff_u$).
By explicit computation (or exponential family identities), we see that the function $g(u)= \log p_{\eta_u}$
possesses derivatives
\begin{align*}
g'(u)&=-\frac t{\sqrt n}(\xi_{\eta_0}^f-F_u\xi_{\eta_0}^f),\qquad
g''(u)=-\frac {t^2}{n}\int (\xi_{\eta_0}^f-F_u\xi_{\eta_0}^f)^2\,dF_u,\\
g'''(u)&=\frac {t^3}{n^{3/2}}\int (\xi_{\eta_0}^f-F_u\xi_{\eta_0}^f)^3\,dF_u.
\end{align*}
The third derivative is bounded by a multiple of $n^{-3/2}$, uniformly in $u$ and $f$.
Since $g''(u)$ is a constant and the empirical process centered, $\sqrt n\G_ng''(0)=0$,
while $\sqrt n\G_ng'(0)= -t\G_n \widetilde{\chi}_{\eta_0}^f$. Next
$nP_{\eta_0}g'(0)=-t\sqrt n(F_0-F)\xi_{\eta_0}^f$, while
$$nP_{\eta_0}g''(0)=-t^2\int (\xi_{\eta_0}^f-F\xi_{\eta_0}^f)^2\,dF
=-t^2\int (\xi_{\eta_0}^f-F_0\xi_{\eta_0}^f)^2\,dF_0+o(1),$$
uniformly in $\{f: \eta\in\mH_n\}$ by assumption \eqref{EqConsistencyf}. Inserting these approximations
together with the Taylor expansion $\log p_\eta - \log p_{\eta_t^f}= -g'(0)-g''(0)/2+O(n^{-3/2})$ in \eqref{eq:lan_least_fav_dir},
with $b$ replaced by $f$, yields \eqref{eq:LAN_remainder_f}.
\end{proof}

Conditions \eqref{EqPriorConcentration}--\eqref{eq:no_bias_f} permit to control the remainder terms
in an expansion of the likelihood. The first four conditions
\eqref{EqPriorConcentration}--\eqref{eq:prob_sup} require that the posterior concentrates on
shrinking neighbourhoods about the true parameters $b_0$ and $f_0$, though not $a_0$, and hence
mostly require consistency, whereas the remaining condition \eqref{eq:no_bias_f} also requires a
$\sqrt n$-rate on a certain bias term. 

In Theorems~\ref{thm:general_F}-\ref{thm:dep_prior_gen_EB}, 
which put a Dirichlet prior on the distribution $F$ rather than a prior on the
density $f$, condition \eqref{eq:no_bias_f} does not appear and hence this might be interpreted as
involving a bias incurred by possibly putting the wrong prior on $F$.
The condition, which seems tied to any prior that
directly models $f$, may be satisfied for reasonable priors if both $b$ and $f$ are sufficiently
smooth, but in the situation where $f$ has low regularity, even correctly calibrating the smoothness
of the prior on $f$ can perform worse than naively using a Dirichlet process.  The condition
provides another example where an infinite-dimensional prior can induce an undesired
bias \cite{freedman1999,ritov2014, knapik2011,castillo2015}. This effect becomes more pronounced as
the covariate dimension increases and can be problematic in even moderate dimensions.

Consider equipping both $\eta^b$ and $\eta^f$ with Gaussian process priors.  Given
independent mean-zero Gaussian processes  $W^b = (W_z^b: z\in [0,1]^d)$ and
$W^f = (W_z^f:z\in[0,1]^d)$, consider the prior
\begin{align}\label{eq:prior_bf}
b(z) &= \Psi(W_z^b ),\\
\label{eq:prior_f}
f(z) &= \frac{e^{W_z^f}}{\int_{[0,1]^d} e^{W_u^f}\, du}.
\end{align}
We write $\varepsilon_n^i$, for $i\in \{ b,f\}$, for the respective contraction rates for the
two parameters $b$ and $f$, and $\H^i$ for the RKHS corresponding to the process $W^i$. 

\begin{proposition}\label{prop:gp_gen_f}
Consider the product Gaussian process prior \eqref{eq:prior_bf}-\eqref{eq:prior_f} on $b$ and
$f$. Let $\varepsilon_n^b\rightarrow 0$ satisfy \eqref{eq:gaus_conc_rate} with respect to the norm
$\|\cdot\|_\infty$  and suppose $\sqrt{n}\varepsilon_n^b\varepsilon_n^f\rightarrow 0$,
where $\varepsilon_n^f\rightarrow 0$ is a rate of contraction in $L^2$ of the posterior distribution of $f$ to $f_0$. Suppose there exist
sequences $\xi_n = (\xi_n^b,\xi_n^f) \in \H^{b} \times \H^{f}$ and
$\zeta_n^b,\zeta_n^f \rightarrow 0$ such that
\begin{equation}\label{eq:RKHS_cond_f}
\begin{split}
\|\xi_n^i - \xi_{\eta_0}^i\|_\infty \leq \zeta_n^i, \quad  \|\xi_n^i\|_{\H^{i}} \leq \sqrt{n} \zeta_n^i, \quad 
\sqrt{n}\varepsilon_n^i \zeta_n^i \rightarrow 0,\qquad  i\in \{b,f\}.
\end{split}
\end{equation}
Suppose further that there exist measurable sets $\mH_n^b$ of functions such that 
$\Pi\bigl(\eta^b\in(\mH_n^b -t\xi_n^b/\sqrt{n}) |X^{(n)}\bigr)\rightarrow^{P_0} 1$ for every $t\in \R$ and
\eqref{eq:prob_sup} holds.
Then the posterior distribution satisfies the semiparametric BvM theorem.
\end{proposition}

\begin{proof}
We verify the conditions of Theorem~\ref{thm:general_f}. Since the likelihood factorizes and we
have a product prior on $b$ and $f$, the posterior is also a product measure. By Lemma~\ref{lem:contrac_prod_b} 
with norm $\|\cdot\|_\infty$, the posterior distribution of $b$ contracts about $b_0$ at rate $\varepsilon_n^b$ in
$L^2$, while the posterior of $f$ contracts to $f_0$ in $L^2$ at rate  $\varepsilon_n^f$, by assumption.
For $\mH_n^b$ the sets as in the statement of the proposition, define
$$\mH_n = \Bigl\{(\eta^a,\eta^b,\eta^f): \eta^b\in\mH_n^b, \|\Psi(\eta^b)-b_0\|_{L^2} \leq \varepsilon_n^b, 
\bigl\|\frac{e^{\eta^f}}{\textstyle{\int} e^{\eta^f}dz}-f_0\bigr\|_{L^2} \leq \varepsilon_n^f \Bigr\}.$$
Then $\Pi(\mH_n|X^{(n)}) \rightarrow^{P_0} 1$ as $n\rightarrow \infty$, by assumption. Furthermore, for any
$\eta\in \mH_n$, by the Cauchy-Schwarz inequality and the assumption
that $f_0$ is bounded away from zero,
$\bigl| \sqrt{n}\int(b-b_0)(f-f_0)dz \bigr|  \lesssim\sqrt{n} \|b-b_0\|_{L^2(F_0)} \|f-f_0\|_{L^2} 
\leq \sqrt{n}\varepsilon_n^b \varepsilon_n^f $, which tends to zero by assumption.
It follows that $\mH_n$ satisfies conditions \eqref{EqPriorConcentration}--\eqref{eq:no_bias_f}
of Theorem~\ref{thm:general_f}.

It remains to verify \eqref{eq:prior_shift_cond1}, which by the prior independence of $b$ and $f$
factorizes in a $b$-term and an $f$-term.  The $b$-term was considered in detail in the
proof of Proposition~\ref{prop:gp_gen_F}.

The $f$-term consists of a prior change of measure for the
exponentiated Gaussian process prior, which is exactly the situation considered in
\cite{castillo2015}. Since $h(f,f_0) \lesssim \|f-f_0\|_{L^2}$, one may localize the posterior to a
Hellinger neighbourhood of radius $\varepsilon_n^f$ as in Proposition~3 of \cite{castillo2015}. The
result then follows from \cite{castillo2015}, under the same conditions. 
\end{proof}

While the contraction rate $\varepsilon_n^b$ for $b\in L_2(F_0)$ is given by the solution of \eqref{eq:gaus_conc_rate},
the similar equation for $f$ in general yields a contraction rate in the
Hellinger distance \cite{vandervaart2008}, but the proposition requires a rate in
$L^2$. If the prior is supported on a fixed $L^\infty$-ball, for instance suitably conditioned
Gaussian process priors \cite{gine2011}, then the Hellinger rate automatically implies the
same rate in $L^2$-distance. For unbounded priors, such as
Riemann-Liouville processes, one may often use regularity properties of the Gaussian process to 
bootstrap a Hellinger rate to one in $L^2$, and thus take $\varepsilon_n^f$ to be
a solution to the analogue of \eqref{eq:gaus_conc_rate} for $f$ (see Proposition~5 of \cite{castillo2015}).

For the concrete cases of the Riemann-Liouville process and finite Gaussian series prior, the preceding proposition
implies the following.

\begin{corollary}\label{cor:RL_f}
Suppose $a_0\in C^\alpha([0,1]^d)$, $b_0 \in C^\beta([0,1]^d)$ and $f_0\in C^\gamma([0,1]^d)$. Consider the prior 
\eqref{eq:prior_bf}-\eqref{eq:prior_f} with $W^b$ and $W^f$ finite Gaussian series 
as in \eqref{eq:series_prior} with truncation parameters $J_{\bar{\beta}}$ and $J_{\bar{\gamma}}$, respectively. 
If $\alpha,\beta,\gamma>d/2$, $d/2<\bar{\beta}<\alpha+\beta-1/2$, $d/2<\bar{\gamma}<\gamma+\beta-1/2$ and
\begin{equation}
\label{EqBetaGamma}
\frac{\beta\wedge\bar{\beta}}{2\bar{\beta}+d} + \frac{\gamma\wedge\bar{\gamma}}{2\bar{\gamma}+d}>\frac{1}{2},
\end{equation}
then the posterior distribution satisfies the semiparametric BvM theorem. Moreover, when $d=1$ the same result holds with Riemann-Liouville processes \eqref{eq:RL_process} with parameters $\bar{\beta}$ and $\bar{\gamma}$ on $b$ and $f$, respectively. 
\end{corollary}

\begin{proof}
We verify the conditions of Proposition~\ref{prop:gp_gen_f}.

\textbf{Series prior} \eqref{eq:series_prior}: Using the form of the concentration
function in the proof of Theorem~4.5 of \cite{vandervaart2008}, we see that \eqref{eq:gaus_conc_rate} is satisfied for
$$\varepsilon_n^b = n^{-\frac{\beta\wedge \bar{\beta}}{2\bar{\beta}+d}} \log n,
\qquad 
\varepsilon_n^f = n^{-\frac{\gamma\wedge \bar{\gamma}}{2\bar{\gamma}+d}} \log n.$$
This yields a posterior contraction rate $\varepsilon_n^f$ for $f$ in Hellinger distance by 
Lemma~\ref{lem:contrac_prod_f}. Since the Gaussian process \eqref{eq:series_prior} is a finite sum of $S$-regular Daubechies wavelets, it takes values in $C^\delta$ for all $\delta<S$ almost surely, so that we may apply Lemma \ref{lem:hell_L2_rate}. In particular, for $\gamma >d/2$ (and taking $S>\max(\gamma,d/2)$ by assumption) there exists a choice of $K_n \to \infty$ satisfying \eqref{eq:hell_L2_rate}, from which we deduce that $\varepsilon_n^f$ is also a contraction rate for $f$ in $L^2$. By \eqref{EqBetaGamma} the given sequences satisfy 
$\sqrt{n}\varepsilon_n^b \varepsilon_n^f \rightarrow 0$. 
Condition~\eqref{eq:prob_sup} is verified in Lemma~\ref{lem:sup_series},
under the assumption that $\beta \wedge \bar{\beta} >d/2$.

It thus remains to establish \eqref{eq:RKHS_cond_f}, the approximation by elements of the RKHS. The $b$-term was considered in detail in the proof of Corollary \ref{cor:RL_F}, where it was shown that \eqref{eq:RKHS_cond_f} with $i=b$ holds if $\beta\wedge\bar{\beta}>[\bar{\beta}-\alpha+d/2]\vee 0$. Together with the condition $\beta \wedge \bar{\beta}>d/2$ needed to verify \eqref{eq:prob_sup} above, this is equivalent to $\alpha,\beta>d/2$ and $d/2 < \bar{\beta}<\alpha+\beta-d/2$.

The same argument for $f$, with  $\xi_{\eta_0}^f = b_0 - \chi(\eta_0)\in C^\beta$, gives the similar requirement that $\gamma \wedge \bar{\gamma} > [\bar{\gamma}-\beta+d/2] \vee 0$. Together with the condition $\gamma>d/2$ needed to apply Lemma \ref{lem:hell_L2_rate} above, this is implied by $\beta,\gamma>d/2$ and $d/2<\bar{\gamma}<\gamma+\beta-d/2$.

\textbf{Riemann-Liouville prior} \eqref{eq:RL_process}: 
The proof follows in much the same way. Using the form of the concentration
function in Theorem~4 of Castillo \cite{castillo2008}, we see that \eqref{eq:gaus_conc_rate} is satisfied for
$$\varepsilon_n^b = n^{-\frac{\beta\wedge \bar{\beta}}{2\bar{\beta}+1}} (\log n)^\kappa,
\qquad 
\varepsilon_n^f = n^{-\frac{\gamma\wedge \bar{\gamma}}{2\bar{\gamma}+1}} (\log n)^{\kappa'},$$
where  $\kappa$ is function of $(\beta,\bar\beta)$ and 
$\kappa'$ is the same function of $(\gamma,\bar\gamma)$, given explicitly in \cite{castillo2008}.
This yields a posterior contraction rate $\varepsilon_n^f$ for $f$ in Hellinger distance by 
Lemma~\ref{lem:contrac_prod_f}. Since a Riemann-Liouville process with parameter $\bar{\gamma}$ takes
values in $C^\delta$ for all $\delta <\bar{\gamma}$ \cite{lifshits2005}, and the same property
holds when adding a polynomial part in \eqref{eq:RL_process}, we can apply Lemma \ref{lem:hell_L2_rate}. In particular, for $\gamma \wedge \bar{\gamma}>1/2$ there exists a choice of $K_n \to \infty$ satisfying \eqref{eq:hell_L2_rate} so that we deduce that $\varepsilon_n^f$ is also a contraction rate for $f$ in $L^2$. By \eqref{EqBetaGamma} the given sequences satisfy 
$\sqrt{n}\varepsilon_n^b \varepsilon_n^f \rightarrow 0$. 
Condition~\eqref{eq:prob_sup} is verified in Lemma~\ref{lem:sup_RL},
under the assumption that $\beta \wedge \bar{\beta} >1/2$. 

It thus remains to establish \eqref{eq:RKHS_cond_f}. The $b$-term was considered in detail in the proof of Corollary \ref{cor:RL_F}, where it was shown that \eqref{eq:RKHS_cond_f} with $i=b$ holds if $\beta\wedge\bar{\beta}>[\bar{\beta}-\alpha+1/2]\vee 0$. Together with the condition $\beta \wedge \bar{\beta}>1/2$ needed to verify \eqref{eq:prob_sup} above, this is equivalent to $\alpha,\beta>1/2$ and $1/2 < \bar{\beta}<\alpha+\beta-1/2$.

The same argument for $f$, with  $\xi_{\eta_0}^f = b_0 - \chi(\eta_0)\in C^\beta$, gives the similar requirement that $\gamma \wedge \bar{\gamma} > [\bar{\gamma}-\beta+1/2] \vee 0$. Together with the condition $\gamma\wedge \bar{\gamma}>1/2$ needed to apply Lemma \ref{lem:hell_L2_rate} above, this is implied by $\beta,\gamma>1/2$ and $1/2<\bar{\gamma}<\gamma+\beta-1/2$.
\end{proof}

The corollary suggests that modelling the density $f$ using a Gaussian process prior works well 
under smoothness conditions on $f$. If $\bar{\beta}=\beta$ and $\bar{\gamma}=\gamma$, so that the prior processes
select the correct smoothness, the conditions in the corollary reduce to $\alpha,\beta,\gamma>d/2$, although it is known that no smoothness of $f$ is required to estimate the functional
$\chi(\eta)$, for example in Corollaries \ref{cor:RL_F} and \ref{cor:RL_dep_F}. The condition $\gamma>d/2$ can be problematic for even moderate
dimensions. The prior for the  parameter $f$ can therefore have a significant impact
and must be carefully chosen.

\section{Technical results}
\label{SectionTechnicalResults}

In this section we present technical results that are used in the proofs of the main results.

The following lemma controls changes in the likelihood under perturbations of $\eta^b$. 

\begin{lemma}\label{lem:lik_b_approx} 
For bounded functions $\xi_n$ and $\xi_0$, $t\in\RR$, a set $A_n$ of measurable functions, some $w_n>0$ 
and $\varepsilon_n, \zeta_n \rightarrow 0$, suppose that 
$$\|\xi_n\|_\infty \leq w_n,\qquad \|\xi_n-\xi_0\|_{L^2(F_0)} \leq \zeta_n,\qquad
\sup_{\eta^b\in A_n}\|\Psi(\eta^b)-\Psi(\eta_0^b)\|_{L^2(F_0)} \leq \varepsilon_n.$$
If 
$n^{-1/2}w_n \rightarrow 0$, $\sqrt{n} \zeta_n \varepsilon_n \rightarrow 0$ and,
for $b_0=\Psi(\eta_0^b)$,
\begin{equation}\label{eq:sup_lem_lik_b_approx}
(1+w_n)\, \sup_{b=\Psi(\eta^b): \eta^b\in A_n} \bigl|\G_n[b-b_0]\bigr|= o_{P_0}(1),
\end{equation}
then
\begin{equation*}
\sup_{\eta^b\in A_n} \left| \ell_n^b(\eta^b - \tfrac{t}{\sqrt{n}} \xi_n) -\ell_n^b(\eta^b - \tfrac{t}{\sqrt{n}} \xi_0) \right| = o_{P_0}(1).
\end{equation*}
\end{lemma}

\begin{proof}
The part $\ell_n^b$ of the full log-likelihood \eqref{EqLogLikelihood} involving only the terms 
$b=\Psi(\eta^b)$ equals
\begin{align*}
\ell_n^b (\eta^b) & = \sum_{i=1}^n \Bigl[R_i Y_i  \log \frac{e^{\eta^b(Z_i)}}{1+e^{\eta^b(Z_i)}} 
+ R_i (1-Y_i)  \log \frac{1}{1+e^{\eta^b(Z_i)}}\Bigr]\\
& = n\P_n\bigl[ry \eta^b -r \varphi(\eta^b)\bigr],
\end{align*}
where $\varphi (\eta) =\log (1+e^{\eta})$. 
We apply this with $\eta^b$ equal to $\eta_{n,t}:=\eta^b - {t}\xi_n/{\sqrt{n}} $ and 
$\eta_t:=\eta^b - {t} \xi_0/{\sqrt{n}}$, take the difference
and Taylor expand $\varphi(\eta_{n,t})$ and $\varphi(\eta_t)$ about $\eta^b$ to third order.
Since $\varphi' = \Psi$, $\varphi '' =\Psi(1-\Psi)$ and $\varphi''' = \Psi(1-\Psi)(1-2\Psi)$, 
we have that $\varphi'(\eta^b)=b$ and $\varphi''(\eta^b)=b(1-b)$ and
the third derivative is uniformly bounded. Consequently,
\begin{align*}
&\ell_n^b (\eta_{n,t}) - \ell_n^b(\eta_t) 
 = n\P_n\bigr[ry (\eta_{n,t} - \eta_t)-r (\varphi (\eta_{n,t}) - \varphi(\eta_t))\bigr]\\
&= n\P_n\bigl[r(y-b) (\eta_{n,t} - \eta_t)\bigr]
- \frac{n}{2} \P_n\bigl[ r b(1-b)\bigl((\eta_{n,t}-\eta^b)^2-(\eta_t-\eta^b)^2\bigr)] +{\rm R},
\end{align*}
where 
$$|{\rm R}|\lesssim n\P_n \bigl[r|\eta_{n,t}-\eta^b|^3 +r|\eta_t-\eta^b|^3\bigr]
\lesssim \frac{w_n+1}{\sqrt n}\P_n(|\xi_n|^2+|\xi_0|^2).$$
The last expression is  $O_{P_0}((w_n+1)/\sqrt n)$ and 
tends to zero by assumption. The first term on the right of the preceding display can be rewritten as
$$t\G_n[r(y-b_0)(\xi_0-\xi_n)]+ t\G_n[r(b_0-b)(\xi_0-\xi_n)] + t\sqrt{n}P_{\eta_0}[r(y-b)(\xi_0-\xi_n)].$$
Here the first term tends to zero in probability, since $\xi_n\rightarrow \xi_0$ in $L^2(F_0)$,
the second tends to zero uniformly in $b\in B_n:=\bigl\{b=\Psi(\eta^b): \eta^b\in A_n\bigr\}$ 
by assumption \eqref{eq:sup_lem_lik_b_approx} and Lemma~\ref{lem:emp_proc_change} 
applied with $\varphi=r(\xi_0-\xi_n)$, which satisfies $\|\varphi\|_\infty =O(1+w_n)$. By computing the expectation
by first conditioning on $z$, next bounding out $a_0^{-1}(z)=\E (R\given Z=z)$ and next applying 
the Cauchy-Schwarz inequality, the third term is bounded above in absolute value by
$|t|\|1/a_0\|_\infty\sqrt{n} \|b-b_0\|_{L^2(F_0)} \|\xi_n-\xi_0\|_{L^2(F_0)} \lesssim \sqrt{n}\varepsilon_n \zeta_n$, 
which also tends to zero by assumption.
The second, quadratic term of the expansion $\ell_n^b (\eta_{n,t}) - \ell_n^b(\eta_t)$ can be rewritten as
$$-\frac{t^2}{2} (\P_n-P_{\eta_0}) [rb(1-b)(\xi_n^2-\xi_0^2)] - \frac{t^2}{2}P_{\eta_0}[rb(1-b)(\xi_n^2-\xi_0^2)].$$
The second term is bounded above in absolute value by
$t^2P_0|\xi_n^2-\xi_0^2|\rightarrow 0$, as $\xi_n\rightarrow \xi_0$ in $L^2(F_0)$.
The first term tends to zero, uniformly in $b\in B_n$ by \eqref{eq:sup_lem_lik_b_approx} and
Lemma~\ref{LemmaGCnPreservation} applied with the classes of functions $\{r\}$, $B_n$ and
$\{\xi_n^2-\xi_0^2\}$ and the continuous function $(r,b,a)\mapsto rb(1-b)a$.  
\end{proof}

\begin{lemma}
[Lemma~1 of \cite{ghosal2007}]
\label{lem:small_prior_prob}
If $B_n$ are measurable sets such that 
$\Pi(\eta\in B_n)/\Pi\bigl(\eta: K\vee V(p_{\eta_0},p_\eta) \leq \varepsilon_n^2\bigr) = o(e^{-2n\varepsilon_n^2})$, 
then $P_0 \Pi(\eta\in B_n|X^{(n)}) \rightarrow 0$. 
\end{lemma}

Since all contraction rates in the paper are established using the testing approach of Ghosal et
al. \cite{ghosal2000}, establishing a contraction rate $\varepsilon_n$ automatically involves
proving a lower bound of the form $\Pi (B_{KL}({\eta_0},\varepsilon_n)) \geq
e^{-Cn\varepsilon_n^2}$. The condition of the lemma
is therefore satisfied if $\Pi(\eta\in B_n)\le e^{-Ln\varepsilon_n^2}$ for sufficiently large $L$.

In the case of a product prior on the parameters $(a,b)$ or $(a,b,f)$ and sets $B_n$ that refer
to only one of the parameters $a, b, f$, both the numerator and denominator in the condition
of the lemma factorize. Consequently, one may use the rate obtained from lower bounding the small-ball
probability for that parameter, rather than the worst rate $\max (\varepsilon_n^a,\varepsilon_n^b)$ or
$\max (\varepsilon_n^a,\varepsilon_n^b,\varepsilon_n^f)$. For instance if we consider the product
prior \eqref{eq:prior_bf}-\eqref{eq:prior_f} and $B_n$ involves only conditions concerning $b$, then
we use the rate $\varepsilon_n^b$ in Lemma~\ref{lem:small_prior_prob}.

These two remarks will be used throughout without further mention.

\begin{lemma}\label{lem:sup_RL}
Let $\Pi$ be the Gaussian process prior $\Psi(R^{\bar\beta})$ on $b$,
for $R^{\bar\beta}$ the Riemann-Liouville process given in \eqref{eq:RL_process}. 
If $b_0 \in C^\beta$ and $\beta \wedge \bar{\beta}>1/2$, then there exist sets $\mH_n^b\subset C[0,1]$ such that 
$\Pi(\eta^b\in \mH_n^b -t\xi_n^b/\sqrt{n} |X^{(n)})\rightarrow^{P_0} 1$ for every $t\in \R$ and every $\xi_n^b$ 
satisfying \eqref{eq:RKHS_cond_b}, and such that \eqref{eq:prob_sup2} holds.
\end{lemma}

\begin{proof}
We shall show that the posterior distribution concentrates on a set of small bracketing entropy,
which will allow us to verify \eqref{eq:prob_sup2} by a standard maximal inequality for
the empirical process. 

The Riemann-Liouville process $R^{\bar{\beta}}$ can be viewed
as a Gaussian random element in the  Sobolev space $H^s$ for $s<\bar\beta$ and
its RKHS is equal to $H^{\bar{\beta}+1/2}$ by Theorem~4.2 of \cite{vandervaart2008}.
For $H_1^s$  the unit ball of $H^s$ and $1/2<s<\bar{\beta}$ to be specified below, we define
\begin{align*}
\mH_n^b = \bigl\{\gamma_n H_1^s + M\sqrt{n}\varepsilon_n^b H_1^{\bar{\beta}+1/2} \bigr\}
\bigcap \bigl\{\eta^b: \|\Psi(\eta^b)-b_0\|_{L^2(F_0)} \leq \varepsilon_n^b\bigr\},
\end{align*}
where $M$ is a large constant to be determined and 
\begin{equation}
\label{EqDefGammaEpsilon}
\gamma_n = n^{-\frac{2(\bar{\beta}+1/2-\beta\wedge\bar{\beta})(\bar{\beta}-s)}{2\bar{\beta}+1}}
(\log n)^{-2\kappa(\bar{\beta}-s)},
\qquad
\varepsilon_n^b\asymp n^{-\frac{\beta\wedge \bar{\beta}}{2\bar{\beta}+1}} (\log n)^\kappa.
\end{equation}
The rate $\varepsilon_n^b$ is the contraction rate of the posterior distribution of $b$.
This is given by \eqref{eq:gaus_conc_rate} in view of Lemma~\ref{lem:contrac_prod_b}, 
and shown to take the form as in the display in \cite{castillo2008} (see his Theorem~4), 
where the exponent $\kappa$ of the logarithmic
factor depends in a complicated manner on $(\beta,\bar\beta)$, but will not be of concern here. 
It follows that the second set in the definition of $\mH_n^b$ has posterior mass tending to one
in probability, provided the proportionality constant in $\varepsilon_n^b$ is chosen large enough. 

By the Borell-Sudakov inequality (e.g.\ Theorem~5.1 in \cite{vandervaart2008b}),
the prior probability of the complement of the first part of $\mH_n^b$ satisfies 
\begin{align*}
\Pi \bigl[ R^{\bar{\beta}} \not\in \gamma_n H_1^s + M\sqrt{n}\varepsilon_n^b H_1^{\bar{\beta}+1/2} \bigr] 
\leq 1 - \Phi \Bigl[\Phi^{-1}\bigl[\Pi(\|R^{\bar{\beta}}\|_{H^s} \leq \gamma_n)\bigr]+M\sqrt{n}\varepsilon_n^b \Bigr].
\end{align*}
Reasoning as in Theorem~4.3.36 of \cite{gine2016}, we see that 
$\log N(H_1^{\bar{\beta}+1/2},\|\cdot\|_{H^s},\varepsilon) \leq K(\bar{\beta},s) \varepsilon^{-1/(\bar{\beta}+1/2-s)}$
for small $\varepsilon >0$. Therefore, by Theorem~1.2 of Li and Linde \cite{li1999},
the small ball exponent $-\log \Pi(\|R^{\bar{\beta}} \|_{H^s}\leq \gamma)$ is
bounded above by a multiple of $\gamma^{-1/(\bar{\beta}-s)}$ for small $\gamma$. 
Substituting this in the preceding display and using the bounds $\Phi^{-1}(y) \geq -\sqrt{2\log(1/y)}$ 
for $0<y<1$, and $1-\Phi(x) \leq e^{-x^2/2}$ for $x>0$, we see that the right
side of the display is  smaller than $e^{-Ln(\varepsilon_n^b)^2}$, where $L$ can be
made large by choosing $M$ large enough.
It then follows from Lemma~\ref{lem:small_prior_prob} that the posterior probability of the
first set in the definition of $\mH_n^b$  also tends to 1 in probability.


For $t\in \R$ and $\xi_n^b$ satisfying \eqref{eq:RKHS_cond_b}, since
$\|\xi_n^b\|_{H^{\bar{\beta}+1/2}} \leq \sqrt{n} \zeta_n^b$ and
$\|\Psi(\eta^b)-\Psi(\eta^b-t\xi_n^b/\sqrt n)\|_{L^2(F_0)}\le |t|\|\xi_n^b\|_\infty/\sqrt n$,
\begin{align*}
\mH_n^b - t\xi_n^b/\sqrt{n} & \supset \bigl\{\gamma_n H_1^s + (M\sqrt{n}\varepsilon_n^b - |t|\zeta_n^b) H_1^{\bar{\beta}+1/2})\bigr\} \\
&\quad \bigcap 
\bigl\{\eta^b: \|\Psi(\eta^b)-b_0\|_{L^2(F_0)}\leq \varepsilon_n^b - |t|(\|\xi_{\eta_0}^b\|_\infty + \zeta_n^b)/\sqrt{n}\bigr\}.
\end{align*}
Since $\sqrt n\varepsilon_n^b\rightarrow\infty$ and $\zeta_n^b \rightarrow 0$, this set hardly differs from the
set $\mH_n^b$. Its posterior probability is seen to tend to 1 in probability by the same arguments
as for $\mH_n^b$, possibly after replacing $\varepsilon_n^b$ with a multiple of itself.


Finally we show  \eqref{eq:prob_sup2}, which is that
the supremum of the empirical process indexed by the class of functions
$\mF_n := \bigl\{ \Psi(\eta^b)-b_0: \eta^b \in \mH_n^b \bigr\}$ tends to zero. 
By Theorem~2.14.2 of \cite{vandervaart1996} applied with the constant function $1$  as envelope function,
\begin{align*}
P_0 \sup_{f\in \mF_n} |\G_n f| 
& \lesssim \int_0^{\delta} \sqrt{1+\log N_{[]}(\mF_n, L^2(P_0),\tau)}\, d\tau \\
&\qquad\qquad+ \varepsilon_n^b \sqrt{1+\log N_{[]}(\mF_n,L^2(P_0),\delta)},
\end{align*}
for any $\delta >0$ such that $\sqrt n\delta> \sqrt{1+\log N_{[]}(\mF_n,L^2(P_0),\delta)}$ (so that
$\sqrt n a(\delta)$ as defined in Theorem~2.14.2 is bigger than the envelope 1).
It therefore remains to bound the above bracketing entropy and pick $\delta=\delta_n\rightarrow0$ 
such that all the terms converge to zero. 

Because $\Psi$ is monotone and Lipschitz, a set of $\tau$-brackets in $L^2(P_0)$ for $\mH_n^b$ 
transforms into a set of $\tau$-brackets in $L^2(P_0)$ for $\mF_n$. Furthermore, separate sets of brackets
for the two constituents of the set
$\gamma_n H_1^s + M\sqrt{n}\varepsilon_n^b H_1^{\bar{\beta}+1/2}$ 
can be combined into brackets for the sum space. 
By \cite{BirmanSolomjak} (or see \cite{vdVW2019}, Section~2.7.2),
\begin{equation}\label{eq:bracketing_entropy}
\begin{split}
\log N_{[]}(\mH_n^b,L^2(P_0),2\tau) 
& \leq \log N(\gamma_n H_1^s,\|\cdot\|_\infty,\tau) \\
&\qquad\qquad+ \log N(M\sqrt{n}\varepsilon_n^b H_1^{\bar{\beta}+1/2},\|\cdot\|_\infty,\tau) \\
& \lesssim (\gamma_n/\tau)^{1/s} + (\sqrt{n}\varepsilon_n^b/\tau)^{1/(\bar{\beta}+1/2)}.
\end{split}
\end{equation}
The first term dominates if 
$\tau \lesssim \delta_*:= n^{-\frac{\bar{\beta}+1/2-\beta\wedge\bar{\beta}}{\bar{\beta}+1/2-s}[\bar{\beta}-s + \frac{s}{2\bar{\beta}+1}]} (\log n)^{-2\kappa \bar{\beta}} \rightarrow 0$. 
For $\delta_n>0$ such that $\delta_*/\delta_n \rightarrow 0$, we therefore have
\begin{align*}
P_0 \sup_{f\in \mF_n} |\G_n f| 
& \lesssim \gamma_n^\frac{1}{2s}\! \int_0^{\delta_*}\!\! \tau^{-\frac{1}{2s}}\, d\tau 
+ (\sqrt{n} \varepsilon_n^b)^\frac{1}{2\bar{\beta}+1}\! \int_{\delta_*}^{\delta_n}\!\! \tau^{-\frac{1}{2\bar{\beta}+1}}\, d\tau 
+ \varepsilon_n^b \Bigl[\frac{\sqrt{n}\varepsilon_n^b}{\delta_n}\Bigr]^{\frac1{2\bar{\beta}+1}},
\end{align*}
provided $\sqrt{n}\delta_n \gtrsim  (\sqrt{n}\varepsilon_n^b/\delta_n)^{1/(2\bar{\beta}+1)}$.
The first integral on the right tends to zero, since
$\gamma_n,\delta_*\rightarrow 0$ and $s>1/2$. 
Thus we must choose $\delta_n$ so that 
\begin{align*}
\delta_n&\gg \delta_*,\\
\delta_n&\ll\delta_{II}:=(\sqrt n\varepsilon_n^b)^{-1/(2\bar\beta)}= n^{-\frac{\bar{\beta}+1/2-\beta\wedge
  \bar{\beta}}{2\bar{\beta}(2\bar{\beta+1)}}} (\log n)^{-\kappa/(2\bar{\beta})},\\
\delta_n&\gg \delta_{IV} := (\varepsilon_n^b)^{2\bar\beta+1}\sqrt n\varepsilon_n^b
=n^\frac{\bar{\beta}+1/2-2(\bar{\beta}+1)(\beta \wedge \bar{\beta)}}{2\bar{\beta}+1}
(\log n)^{2\kappa(\bar{\beta}+1)},\\
\delta_n&\gg \delta_{III} :=n^{-\frac{2\bar{\beta}^2+\bar{\beta}+\beta \wedge \bar{\beta}}{2(\bar{\beta}+1)(2\bar{\beta}+1)}}(\log
n)^\frac{\kappa}{2(\bar{\beta}+1)}.
\end{align*}
The middle two restrictions arise from the second integral and the last term on the right
side of the preceding display, and the fourth is a strengthening of 
the restriction $\sqrt{n}\delta_n\gtrsim  (\sqrt{n}\varepsilon_n^b/\delta_n)^{1/(2\bar{\beta}+1)}$. 
Necessary and sufficient conditions for the four requirements  are: (i) $\delta_* = o(\delta_{II})$, (ii)
$\delta_{III}=o(\delta_{II})$ and (iii) $\delta_{IV} = o(\delta_{II})$. Substituting in the values
of $\delta_*,\delta_{II},\delta_{III},\delta_{IV}$, using that $1/2<s<\bar{\beta}$ and rearranging,
we find the necessary and sufficient conditions: (i) $\bar{\beta}>1/2$, (ii)
$\bar{\beta}>1/2$ and (iii) $\beta\wedge \bar{\beta} >1/2$. Thus under the condition
$\beta \wedge \bar{\beta}>1/2$, we can find a choice of $\delta_n$ such that all requirements are met.
\end{proof}

\begin{lemma}\label{lem:sup_series}
Let $\Pi$ be the Gaussian process prior $\Psi(W^b)$ on $b$ with $W^b$ the finite Gaussian series given in \eqref{eq:series_prior} with truncation parameter $J_{\bar{\beta}}$. If $b_0 \in C^\beta([0,1]^d)$ and $\beta \wedge \bar{\beta}>d/2$, then there exist sets $\mH_n^b\subset C([0,1]^d)$ such that 
$\Pi(\eta^b\in \mH_n^b -t\xi_n^b/\sqrt{n} |X^{(n)})\rightarrow^{P_0} 1$ for every $t\in \R$ and every $\xi_n^b$ 
satisfying \eqref{eq:RKHS_cond_b}, and such that \eqref{eq:prob_sup2} holds.
\end{lemma}

\begin{proof}
As in Lemma \ref{lem:sup_RL}, we show that the posterior distribution concentrates on a set of small bracketing entropy,
which will allow us to verify \eqref{eq:prob_sup2} by a standard maximal inequality for the empirical process. Write $J = J_{\bar{\beta}}$ and define $V_J = \text{span}(\psi_{jk}: j\leq J, k)$, where $\text{dim}(V_J) = O(2^{Jd}) = O(n^{d/(2\bar{\beta}+d)})$. Recall that the RKHS of the Gaussian series prior \eqref{eq:series_prior} is
$$\H^b = \Bigl\{ w\in V_J: \|w\|_{\H^b}^2 := \sum_{j\leq J} \sum_k \sigma_j^{-2} |\langle w,\psi_{jk}\rangle_{L^2}|^2 < \infty \Bigr\}.$$
Define the set
$$\mH_n^b = \bigl\{ w\in V_J: \|w\|_{\H^b} \leq M \sqrt{n}\eps_n^b,~ \|\Psi(w)-b_0\|_{L^2(F_0)} \leq \eps_n^b \bigr\},$$
where $M$ is a large constant to be determined and $\varepsilon_n^b\asymp n^{-\frac{\beta\wedge \bar{\beta}}{2\bar{\beta}+d}} (\log n)$. The rate $\varepsilon_n^b$ is the contraction rate of the posterior distribution of $b$. This is given by \eqref{eq:gaus_conc_rate} in view of Lemma~\ref{lem:contrac_prod_b}, and shown to take this form in Theorem 4.5 of \cite{vandervaart2008}. It follows that the second condition in the definition of $\mH_n^b$ is satisfied with posterior mass tending to one in probability, provided the proportionality constant in $\varepsilon_n^b$ is chosen large enough.

By Hilbert space duality, one can write $\|w\|_{\H^b} = \sup_{\|u\|_{\H^b} \leq 1} \langle w,u\rangle_{\H^b}$, and since $\H^b$ is separable, one can further restrict the supremum to a countable subset $\mathcal{U} \subseteq \{u:\|u\|_{\H^b} \leq 1\}$. Note that for 
$W^b\sim \Pi$ as in \eqref{eq:series_prior}, $\E\|W^b\|_{\H^b}^2 = \sum_{j\leq J}\sum_k 1 \lesssim 2^{Jd}$ and $\sup_{u\in \mathcal{U}} \E\langle W^b,u\rangle_{\H^b}^2= \sup_{u\in\mathcal{U}} \|u\|_{\H^b}^2 \leq 1.$ Thus by Borell's inequality for the supremum of a Gaussian process (\cite{ledoux2001}, p. 134), $\Pi \left( \|W^b\|_{\H^b} \geq C2^{Jd/2} + x \right) \leq e^{-x^2/2}$ for any $x > 0$. 
Setting $x=(M/2)\sqrt{n}\eps_n^b$ and using that $2^{Jd/2} = o(\sqrt{n}\eps_n^b)$, this yields $\Pi(\|W^b\|_{\H^b} \geq M\sqrt{n}\eps_n^b) \leq e^{M^2 n(\eps_n^b)^2/8}$ for $n$ large enough. It then follows from Lemma~\ref{lem:small_prior_prob} that the first condition in the definition of $\mH_n^b$ is also satisfied with posterior probability tending to 1 in $P_0$-probability for $M>0$ large enough.

For $t\in \R$ and $\xi_n^b\in V_J$ satisfying \eqref{eq:RKHS_cond_b}, since $\|\xi_n^b\|_{\H^b} \leq \sqrt{n} \zeta_n^b$ and $\|\Psi(\eta^b)-\Psi(\eta^b-t\xi_n^b/\sqrt n)\|_{L^2(F_0)}\le |t|\|\xi_n^b\|_\infty/\sqrt n$,
\begin{align*}
\mH_n^b- t\xi_n^b/\sqrt{n}  & \supset \bigl\{ w\in V_J: \|w\|_{\H^b} \leq M\sqrt{n}\eps_n^b - |t|\zeta_n^b,\\
&\quad\quad \quad  \|\Psi(w)-b_0\|_{L^2(F_0)}\leq \varepsilon_n^b - |t|(\|\xi_{\eta_0}^b\|_\infty + \zeta_n^b)/\sqrt{n}\bigr\}.
\end{align*}
Since $\sqrt n\varepsilon_n^b\rightarrow\infty$ and $\zeta_n^b \rightarrow 0$, this set hardly differs from the
set $\mH_n^b$. Its posterior probability is seen to tend to 1 in probability by the same arguments
as for $\mH_n^b$, possibly after replacing $\varepsilon_n^b$ with a multiple of itself.

We now show  \eqref{eq:prob_sup2}, which is that the supremum of the empirical process indexed by the class of functions
$\mF_n := \bigl\{ \Psi(w)-b_0: w \in \mH_n^b \bigr\}$ tends to zero. We first establish the following bound on the bracketing entropy of $\mF_n$:
\begin{align}\label{brackets}
\log N_{[]}(\mF_n,L^2(P_0),\tau) \lesssim 2^{Jd} \log (n/\tau).
\end{align}
Since $\Psi$ is monotone and Lipschitz, a set of $\tau$-brackets in $L^2(P_0)$ for $\mH_n^b$ 
transforms into a set of $\tau$-brackets in $L^2(P_0)$ for $\mF_n$. Note that for $v,w\in V_J$, by Cauchy-Schwarz and since $\sup_z \sum_k |\psi_{jk}(z)|^2 \lesssim 2^{jd}$ for all $j$ by the localization property of the wavelets,
\begin{align*}
\|v-w\|_\infty & = \sup_z \left| \sum_{j\leq J} \sum_k \langle v-w,\psi_{jk}\rangle_{L^2} \psi_{jk}(z) \right| \\
& \leq  \left( \sum_{j\leq J}\sum_k \sigma_j^{-2}\langle v-w,\psi_{jk}\rangle_{L^2}^2 \right)^{1/2} \sup_z \left( \sum_{j\leq J} \sum_k \sigma_j^2 |\psi_{jk}(z)|^2 \right)^{1/2} \\
& \lesssim \|v-w\|_{\H^b} \left( \sum_{j\leq J}  2^{-2jr} \right)^{1/2} \lesssim \|v-w\|_{\H^b}.
\end{align*}
Therefore,
\begin{align*}
\log N_{[]}(\mH_n^b,L^2(P_0),2\tau) &\leq  \log N(\mH_n^b,\|\cdot\|_\infty,\tau)\\
& \leq \log N(\{w\in V_J: \|w\|_{\H^b} \leq M \sqrt{n}\eps_n^b\},\|\cdot\|_{\H^b},c\tau)
\end{align*}
for some $c>0$. Recall that for any norm on $\R^m$, the unit ball in that norm has $\tau$-covering number bounded by $(3/\tau)^m$ (\cite{gine2016}, Proposition 4.3.34). The right hand side of the last display is thus be bounded by $\text{dim}(V_J) \log (3cM\sqrt{n}\eps_n^b /\tau) \lesssim 2^{Jd} \log (n/\tau)$, which proves \eqref{brackets}.

By Theorem~2.14.2 of \cite{vandervaart1996} applied with the constant function $1$ as envelope function and \eqref{brackets},
\begin{align*}
P_0 \sup_{f\in \mF_n} |\G_n f| 
& \lesssim \int_0^{\delta} 2^{Jd/2} \sqrt{\log (n/\tau)}\, d\tau + \varepsilon_n^b 2^{Jd/2} \sqrt{\log (n/\delta)}
\end{align*}
for any $\delta >0$ such that $\sqrt n\delta> \sqrt{1+\log N_{[]}(\mF_n,L^2(P_0),\delta)}$ (so that
$\sqrt n a(\delta)$ as defined in Theorem~2.14.2 is bigger than the envelope 1). By \eqref{brackets} and since $2^{Jd} \sim n^{d/(2\bar{\beta}+d)}$, this last condition is satisfied for $\delta_n = M_0n^{-\bar{\beta}/(2\bar{\beta}+d)} \sqrt{\log n}$ with $M_0$ large enough. Recall the inequality
\begin{align*}
\int_0^a \sqrt{\log (A/x)}dx \leq \frac{2\log A}{2\log A-1} a\sqrt{\log (A/a)} \leq 4a\sqrt{\log (A/a)}
\end{align*}
for any $A\geq 2$ and $0<a\leq 1$ (p. 190 of \cite{gine2016}). Using this inequality in the second last display and setting $\delta = \delta_n = o(\eps_n^b)$,
\begin{align*}
P_0 \sup_{f\in \mF_n} |\G_n f| 
& \lesssim 2^{Jd/2} \sqrt{\log (n/\delta_n)}(\delta_n + \varepsilon_n^b) \lesssim (\log n)^{3/2} n^{\frac{d/2-\beta\wedge\bar{\beta}}{2\bar{\beta}+d}},
\end{align*}
which tends to zero for $\beta\wedge\bar{\beta}>d/2$.
\end{proof}

\begin{lemma}\label{lem:sup_RL_dep}
Let the prior $\Pi$ on $b$ be the distribution of $\Psi(R^{\bar\beta}+\lambda a_n)$ 
with  $R^{\bar\beta}$ the Riemann-Liouville process \eqref{eq:RL_process},
independent of $\lambda\sim N(0,\sigma_n^2)$ and $a_n$ a sequence of functions
with $\|a_n\|_\infty=O(1)$. If $b_0 \in C^\beta$ and $\beta \wedge \bar{\beta}>1/2$, then there exist 
sets $\mH_n^b\subset C[0,1]$ such that 
$\Pi ((w,\lambda): w+(\lambda + tn^{-1/2})a_n \in \mH_n^b|X^{(n)}) \rightarrow^{P_0} 1$
for every $t\in \R$, and such that \eqref{eq:prob_sup3} holds.
\end{lemma}

\begin{proof}
Since the proof of this result is very similar to that of Lemma~\ref{lem:sup_RL}, we present only
an outline. For $1/2<s<\bar{\beta}$, define the set
$\mH_n^b=\{w+\lambda a_n: (w,\lambda)\in \mW_n\}$, for
\begin{align*}
\mW_n 
& = \bigl\{(w,\lambda): w\in\gamma_n H_1^s + M\sqrt{n}\varepsilon_n^b H_1^{\bar{\beta}+1/2}, 
|\lambda| \leq M\sigma_n \sqrt{n} \varepsilon_n^b \bigr\} \\
& \qquad\qquad\qquad \qquad  \bigcap 
\bigl\{(w,\lambda) :\|\Psi(w + \lambda a_n)-b_0\|_{L^2(F_0)} \leq \varepsilon_n^b\bigr\},
\end{align*}
where $\varepsilon_n^b$ and $\gamma_n$ are defined in \eqref{EqDefGammaEpsilon}
in the proof of Lemma~\ref{lem:sup_RL}. The first set in the intersection that defines
$\mW_n$ is seen to have posterior probability tending to one
by the same argument as in Lemma~\ref{lem:sup_RL}, combined with 
the univariate Gaussian tail inequality 
$\Pi(|\lambda| \geq M \sigma_n \sqrt{n}\varepsilon_n^b)\leq 2e^{-Mn(\varepsilon_n^b)^2/2}$.
The posterior probability of the second set in the intersection tends to one in probability 
by Lemma~\ref{lem:contrac_dep_EB}. 
The posterior probability of the set $\mW_n$ shifted in its second coordinate
is seen to tend to 1 in the same manner. These results are true uniformly in 
$\| a_n\|_\infty\le M$.


To verify \eqref{eq:prob_sup3} we apply the bracketing maximal inequality,
Theorem~2.14.2 of \cite{vandervaart1996}, as in the proof of Lemma~\ref{lem:sup_RL}.
Presently we need to control the $L^2(P_0)$-bracketing entropy of 
the sets $\mF_n = \bigl\{ \Psi(w+\lambda a_n)-b_0: (w,\lambda)\in \mW_n \bigr\}$. 
Since
\begin{align*}
N_{[]}(\mF_n,L^2(P_0),3\tau) & \leq  N(\gamma_n H_1^s,\|\cdot\|_\infty,\tau) \,
 N(M\sqrt{n}\varepsilon_n^b H_1^{\bar{\beta}+1/2},\|\cdot\|_\infty,\tau) \\
& \qquad 
\times N\bigl([-M\sigma_n\sqrt{n}\varepsilon_n^b ,M\sigma_n\sqrt{n} \varepsilon_n^b ],|\cdot|,\tau/\|a_n\|_\infty\bigr),
\end{align*}
and the last term is of strictly smaller order than the second, we recover the right-hand side of
\eqref{eq:bracketing_entropy} and hence the proof can be completed as in Lemma~\ref{lem:sup_RL}.
\end{proof}

\begin{lemma}\label{lem:sup_series_dep}
Let the prior $\Pi$ on $b$ be the distribution of $\Psi(W+\lambda a_n)$ with $W$ 
the finite Gaussian series \eqref{eq:series_prior},
independent of $\lambda\sim N(0,\sigma_n^2)$ and $a_n$ a sequence of functions
with $\|a_n\|_\infty=O(1)$. If $b_0 \in C^\beta([0,1]^d)$ and $\beta \wedge \bar{\beta}>d/2$, then there exist 
sets $\mH_n^b\subset C([0,1]^d)$ such that 
$\Pi ((w,\lambda): w+(\lambda + tn^{-1/2})a_n \in \mH_n^b|X^{(n)}) \rightarrow^{P_0} 1$
for every $t\in \R$, and such that \eqref{eq:prob_sup3} holds.
\end{lemma}

\begin{proof}
The proof follows that of Lemma \ref{lem:sup_series} with the same modifications as in Lemma \ref{lem:sup_RL_dep} and is hence omitted.
\end{proof}

\section{Auxiliary results}\label{sec:auxiliary_results}

\subsection{Empirical process results}
In this subsection we collect some novel results on empirical processes, which are useful
to simplify otherwise long lists of technical conditions for the main results.
We denote by $\P_n$ and $\G_n=\sqrt n (\P_n-P_0)$ the empirical measure and process of a
sample of i.i.d.\ observations $X_1,\ldots,X_n$ with law $P_0$ in a sample space $(\X,\A)$.  For the
interpretation of outer expectation and probability, these are understood to be defined as the
coordinate projections on a product probability space.

\begin{lemma}\label{lem:emp_proc_change}
For any set $\mH$ of measurable functions $h:\X\to\RR$ and bounded measurable function $\varphi: \X\to[0,1]$, 
\begin{align*}
\E^*\sup_{h \in \mH} |\G_n[\varphi h]| 
\leq 4 \|\varphi\|_\infty \E^* \sup_{h\in \mH}|\G_n[h]|+  \sqrt{P_0(\varphi-P_0\varphi)^2} \sup_{h\in \mH} |P_0h| .
\end{align*}
\end{lemma}

\begin{proof}
Since $\G_n[\varphi h]= \G_n[\varphi (h-P_0h)]+P_0h\G_n\varphi$, and the expectation of $\sup_h|P_0h|\,| \G_n\varphi|$
is bounded above by the second term on the right of the lemma, it suffices to bound
$\E^*\sup_h\bigl|\G_n[\varphi (h-P_0h)]\bigr|$ by the first term on the right of the lemma.
The latter term does not change if every $h$ is replaced by $h-P_0h$. Therefore, it suffices to prove
the lemma under the assumption that $P_0h=0$, for every $h$.

Let $\varepsilon_1,\dots,\varepsilon_n$ be i.i.d. Rademacher random variables independent of the
observations, defined on an additional factor of the underlying probability space that carries the observations. 
By the symmetrization inequality Lemma~2.3.6 of \cite{vandervaart1996}, 
$$\E^*\sup_h \Bigl| \sum_{i=1}^n \bigl(\varphi (X_i) h(X_i) - P_0[\varphi h]\bigr) \Bigr| 
\leq 2\|\varphi\|_\infty \E^*\sup_h \Bigl| \sum_{i=1}^n \varepsilon_i \frac{\varphi (X_i)}{\|\varphi\|_\infty} h(X_i)  \Bigr|.$$
By the contraction principle, Proposition~A.1.10 of \cite{vandervaart1996},
this inequality remains true if the variables $\varphi(X_i)/\|\varphi\|_\infty$ in the right side are removed.
Since $P_0h=0$ for every $h$, the resulting expression is bounded above by 
$4\|\varphi\|_\infty \E^*\sup_h \bigl| \sum_{i=1}^n  h(X_i) \bigr|$,
by the other part of the symmetrization inequality Lemma~2.3.6 of \cite{vandervaart1996}.
Dividing everything through by $\sqrt{n}$ completes the proof.
\end{proof}

\begin{lemma}\label{lem:prob_to_expect}
If $\mH_n$ are classes of measurable functions with uniformly square-integrable envelope functions,
then $\sup_{h\in \mH_{n}}|\G_n[h]|\rightarrow 0$ in outer probability if and only if the
convergence is in outer mean.
\end{lemma}

\begin{proof}
Convergence in outer mean always implies convergence in outer probability. It suffices
to show the converse in this special case. 
By the Hoffmann-J\o rgensen inequality for moments applied to the stochastic processes $n^{-1/2}(h(X_i)-P_0h)$
(see the second inequality in Proposition~A.1.5 of \cite{vandervaart1996}), this follows if
$\E^* \max_{i\le n}\sup_h |h(X_i)-P_0h|=o(\sqrt n)$. 
If $H_n$ are measurable square-integrable envelope functions, then clearly $\sup_h|P_0h|\le P_0H_n= O(1)$,
while $\max_{i\le n}\sup_h |h(X_i)|\le \max_{i\le n}H_n(X_i)$. For every $\epsilon>0$, the latter variable 
is bounded above by $\epsilon\sqrt n+\sum_{i=1}^nH_n(X_i)1\{H_n(X_i)>\epsilon\sqrt n\}$. 
The expectation of the second term is bounded above by $\sqrt n\epsilon^{-1} P_0 H_n^21_{H_n>\epsilon \sqrt n}=o(\sqrt n)$,
by the assumed uniform integrability of the envelope functions $H_n$. Since this is true for any $\epsilon>0$,
the desired result follows.
\end{proof}

\begin{lemma}
\label{LemmaGCnPreservation}
For $i=1,2,\ldots, k$, let $\mH_{n,i}$ be classes of measurable functions that  
are separable as subsets of $L^1(P_0)$, have uniformly integrable envelope functions ($n=1,2,\ldots$),
and are such that $\sup_{h\in \mH_{n,i}}|(\P_n-P_0)[h]|\rightarrow 0$ in outer probability. Let
$\phi: \RR^k\to\RR$ be continuous. If the classes of functions
$\phi(\mH_{n}):=\{\phi(h_1,\ldots, h_k): h_i\in \mH_{n,i}\}$ have uniformly integrable envelope functions,
then $\E^*\sup_{h\in \phi(\mH_{n})}|(\P_n-P_0)[h]|\rightarrow 0$.
\end{lemma}

\begin{proof}
Because the classes $\mH_{n,i}$ are separable as subsets of $L^1(P_0)$,  they admit
by a result of Talagrand (see Theorem~2.3.16 in \cite{vandervaart1996})  pointwise separable modifications
$\tilde \mH_{n,i}$ relative to the $L^1(P_0)$-norm. Arguing as in the proof of
Theorem~2.10.6 of \cite{vandervaart1996} (or Theorem~2.10.1 in the second edition of this reference),
we see that it suffices to prove the theorem for these separable versions. Thus we may assume that
the classes $\mH_{n,i}$ are appropriately measurable. 

By the preceding lemma and the assumed uniform integrability of the envelope
functions, the convergence in probability $\sup_{h\in \mH_{n,i}}|(\P_n-P)[h]|\rightarrow 0$ 
can be strengthened to convergence in outer mean. 
Then, by the (de)symme\-tri\-zation inequality, Lemma~2.3.6 of \cite{vandervaart1996},
for independent Rademacher variables $\epsilon_1,\ldots, \epsilon_n$,
$$\E^* \sup_{h\in \mH_{n,i}} \Bigl| {1\over n}\sum_{i=1}^n \epsilon_i (h(X_i) - P_0h) \Bigr|
\le 2\E^* \sup_{h\in \mH_{n,i}}|(\P_n-P_0)[h]| \rightarrow 0.$$
Because $\E\sup_ h|n^{-1}\sum_{i=1}^n \epsilon_i P_0h|\le \E |n^{-1}\sum_{i=1}^n \epsilon_i|\,\sup_h|P_0h|$ 
tends to zero in mean by the law of large numbers, we also obtain convergence
to zero of the left side without the terms $P_0h$. Next
by the multiplier inequalities, Lemma~2.9.1 of \cite{vandervaart1996},
we see that, for any $n_0\in\mathbb{N}$ and i.i.d.\ standard normal variables $\xi_1,\ldots, \xi_n$
independent from $X_1,\ldots, X_n$,
$$\E^*\!\sup_{h\in\mH_{n,i}}\Bigl|\frac1n\sum_{i=1}^n \xi_i h(X_i)\Bigr|
\lesssim \frac{n_0 P_0H_{n,i}}n \E \max_{i\le n}|\xi_i|
+\max_{n_0\le k\le n} \E^*\!\sup_{h\in\mH_{n,i}}\Bigl|\frac1k\sum_{i=n_0}^k \epsilon_i h(X_i)\Bigr|.$$
By Jensen's inequality the last term on the right does not decrease if the sum on the right starts at $i=1$
rather than $i=n_0$. Then by the preceding the limsup as $n\rightarrow\infty$ is arbitrarily close to zero
if $n_0$ is large enough. For fixed $n_0$ the first term on the right tends to zero as $n\rightarrow\infty$.
We conclude that the left side tends to zero as $n\rightarrow\infty$.
Given $X_1,\ldots, X_n$ the process
$n^{-1/2}\sum_{i=1}^n\xi_i h(X_i)$ is  mean-zero Gaussian,
with natural pseudo-metric given by the $L^2(\P_n)$-norm.
By Sudakov's inequality (e.g.\ Proposition~A.2.5 of \cite{vandervaart1996}),
$${1\over \sqrt n} \sup_{\epsilon > 0}\epsilon \sqrt{\log N(\mH_{n,i}, L^2 (\P_n),\epsilon)}
\le 3 \E_{\xi} \sup_{h\in \mH_{n,i}}\Bigl|{1\over n} \sum_{i=1}^n \xi_i h(X_i)\Bigr|.$$
Taking the outer expectation across this inequality, we see that the left
side tends to zero in outer expectation and hence 
 $n^{-1}{\log N\bigl(\mH_{n,i}, L^2 (\P_n),\epsilon\bigr)}$ 
tends to zero in outer probability, for every $\epsilon>0$. 

Let $H_{n,1},\ldots, H_{n,k}$ and $F_n$ be integrable envelopes
for the classes of functions $\mH_{n,1},\ldots, \mH_{n,k}$ and $\mF_n:=\phi(\mH_n)$,
respectively, and set $H_n =H_{n,1} \vee \cdots \vee H_{n,k}$. 
Furthermore, for $M\in (0,\infty)$ define $\mF_{n,M}$ to be the class of functions
$f1_{H_n \le M}$, when $f$ ranges over $\mF_n$. Then
$$\sup_{f\in\mF_n}| (\P_n -P_0)(f) | \le (\P_n + P_0)F_n 1_{[H_n > M]}
+ \sup_{f\in\mF_{n,M}}| (\P_n - P_0) f |.$$
The expectation of the first term
on the right converges to $0$ as $M\rightarrow\infty$ by the assumed uniform integrability of
$F_n$ and $H_n$. We shall show that the second term tends to zero in outer mean, for every fixed $M$.  
By Lemma~\ref{LemmaContractionContinuous} below, applied
to the function $\phi: [-M, M]^k \rightarrow \RR$ and $\|\cdot\|$
the $L^1$-norm on $\RR^k$, there exists for every $\epsilon>0$
a $\delta>0$, depending only on $\epsilon$, $\phi$ and $M$, such that
for any pairs $(g_i, h_i ) \in \mH_{n,i}$, the inequality
$$\P_n | g_i - h_i |1_{[H_{n,i}\le M ]} \le {\delta \over k}, \qquad i = 1,\ldots, k$$
implies that 
$$\P_n\bigl| \phi(g_1, \ldots,g_k) - \phi(h_1, \ldots,h_k) | 1_{H_n \le M} \le \epsilon.$$
We conclude that 
$$N\bigl(\mF_{n,M}, L^1(\P_n),\epsilon \bigr) 
\le \prod_{i=1}^k N\Bigl(\mH_{n,i},L^1 (\P_n ),{\delta \over k}\Bigr).$$
Take the logarithm and divide by $n$ to see that
$n^{-1}\log N\bigl(\mF_{n,M}, L^1 (\P_n),\epsilon\bigr)$ tends to zero in outer probability, in view of the
preceding paragraph. This is true for every $\epsilon >0$ and $M>0$. 

If $\mF'_{n,M}$ is a minimal $\epsilon$-net over $\mF_{n,M}$ for the $L^1(\P_n)$-norm and fixed observations,
then
$$\E_\epsilon^*\sup_{f\in \mF_{n,M}}\Bigl|\frac1n\sum_{i=1}^n\epsilon_if(X_i)\Bigr|
\le \E_\epsilon^*\sup_{f\in \mF_{n,M}'}\Bigl|\frac 1n\sum_{i=1}^n\epsilon_if(X_i)\Bigr|+\epsilon.$$
For fixed observations the process $n^{-1/2}\sum_{i=1}^n\epsilon_if(X_i)$ is subgaussian
relative to the $L^2(\P_n)$-norm. For $f\in\mF_{n,M}$ these norms are bounded above by
$M$. Therefore, by the subgaussian maximal inequality the
preceding display is bounded above by a multiple of
$$\sqrt{\frac 1n}\sqrt{1+\log N\bigl(\mF_{n,M}, L^1 (\P_n),\epsilon\bigr)} M+\epsilon.$$
Since this tends to $\epsilon$ in outer probability, for any $\epsilon>0$, the left side of the preceding display
tends to zero in outer probability.
Since this is bounded above by $M$, it follows that its expectation also tends to zero and hence
in view of the symmetrization inequality, Lemma~2.3.6 of \cite{vandervaart1996},
$$\E^*\sup_{f\in \mF_{n,M}}\Bigl|\frac1n\sum_{i=1}^n\bigl(f(X_i)-P_0f\bigr)\Bigr|\le 
2\E^*\sup_{f\in \mF_{n,M}}\Bigl|\frac1n\sum_{i=1}^n\epsilon_if(X_i)\Bigr|\rightarrow 0.$$
This concludes the proof.
\end{proof}

The assumption that the classes $\mH_{n,i}$ are separable as subsets of $L^1(P_0)$ is made
only to ensure enough  measurability. The proof of the preceding lemma is based on a combination of
the proof of the converse of the Glivenko-Cantelli theorem and the preservation result 
\cite{vdVWPreservation}. The lemma extends a known result for Glivenko-Cantelli classes
to sequences of classes that depend on $n$. 

The following lemma is proved in \cite{vdVWPreservation}.
Let $\|\cdot\|$ be any norm on $\RR^k$.

\begin{lemma}
\label{LemmaContractionContinuous}
Suppose that $K \subset \RR^k$ is compact and $\phi: K\to\RR$ is continuous.
Then for every $\epsilon >0$ there
exists $\delta >0$ such that for all $n$ and for all 
$a_1,\ldots,a_n, b_1,\ldots,b_n\in K$, the inequality
$n^{-1} \sum_{i=1}^n \|a_i -b_i\| < \delta $ implies that 
$n^{-1}\sum_{i=1}^n\bigl| \phi(a_i) - \phi(b_i)\bigr| < \epsilon$.
\end{lemma}

\subsection{Miscellaneous results}
In this section we state slightly different versions of otherwise known results.

\begin{lemma}
[Cameron-Martin]
\label{LemCM}
If $W$ is a Gaussian random variable in a separable Banach space with RKHS $\H$, then for
$\xi\in \H$, the distribution $P^{W-\xi}$ of $W-\xi$ is absolutely continuous with respect to the
distribution of $P^W$ with Radon-Nikodym derivative of the form
$$\frac{dP^{W-\xi}}{dP^W}(W)=e^{U(W,\xi)-\|\xi\|_\H^2/2},$$
for a measurable map $w\mapsto U(w,\xi)$ such that
$U(W-g,\xi)\sim N\big(\langle g,\xi\rangle_\H,\|\xi\|_\H^2\bigr)$, for any $g,\xi\in \H$.
\end{lemma}

\begin{proof}
For the statement with $g=0$, giving the representation of the
Radon-Nikodym density and the distributional result
$U(W,\xi)\sim N\big(0,\|\xi\|_\H^2\bigr)$ for $\xi\in\H$, see e.g.\ Section 3 of \cite{vandervaart2008b}. 
For $g\not=0$, we
first note that by this known case applied with $g$ instead of $\xi$, a change of measure gives that
$\Pr(U(W-g,\xi)\in B)=\E [1\{U(W,\xi)\in B\}e^{U(W,g)-\|g\|_\H^2/2}]$. Therefore
$\E e^{tU(W-g,\xi)}=\E \bigl[e^{tU(W,\xi)+U(W,g)-\|g\|_\H^2/2}]$, which can be computed
to be $e^{t^2\|\xi\|_\H^2+t\langle g,\xi\rangle_\H}$ from the joint multivariate distribution
of the variables $U(W,\xi')$ for $\xi'\in\H$, which follows from the Cram\'er-Wold device
and the given normal distribution of $U(W,\xi')$, for every $\xi'$.
\end{proof}


\begin{lemma}
[Laplace transform]
\label{LemmaLaplaceTransform}
Let $T\subset \RR$ contain both a strictly increasing 
and a strictly decreasing sequence of numbers with limit 0.
\begin{itemize}
\item[(i)] If $Y_n$ are random variables with $\E e^{tY_n}\ra e^{t^2\sigma^2/2}$ for every $t\in T$, 
then $Y_n$ tends in distribution to $N(0,\sigma^2)$.
\item[(ii)] If $(Y_n,Z_n)$ are random vectors with $\E (e^{tY_n}\given Z_n)\ra e^{t^2\sigma^2/2}$ in
probability for every $t\in T$, then $d_{BL}\bigl(\L(Y_n\given Z_n),N(0,\sigma^2)\bigr)\ra 0$ in
probability.  
\item[(iii)] If the convergence in the preceding assumption is in the almost sure sense, then the
conclusion is also true in the almost sure sense.
\end{itemize}
\end{lemma}

\begin{proof}
(i) Let $a<0$ and $b>0$ be contained in $T$.  Because $\E e^{tY_n}$
is bounded in $n$ for both $t=a$ and $t=b$, the sequence $Y_n$ is tight by Markov's
inequality. For every $t\in T$ strictly between $a$ and $b$, some power bigger than 1 of the
variables $e^{tY_n}$ are bounded in $L^1$, and hence the sequence $e^{tY_n}$ is uniformly
integrable. Consequently, if $Y$ is a weak limit point of $Y_n$, then $\E e^{tY_n}$ tends to
$\E e^{tY}$ along the same subsequence for every $t\in (a,b)\cap T$. In view of the assumption of
the lemma, it follows that $\E e^{tY}=e^{t^2\sigma^2/2}$. The set $(a,b)\cap T$ is infinite by
assumption. Finiteness of $\E e^{tY}$ on this set implies that the function $z\mapsto \E e^{z Y}$
is analytic in an open strip containing the real axis.  By analytic continuation it is equal to
$e^{z^2\sigma^2/2}$, whence $\E e^{is Y}=e^{-s^2\sigma^2/2}$ for every $s\in\RR$.

(ii) It suffices to show that every subsequence of $\{n\}$ has a
further subsequence with $d_{BL}\bigl(\L(Y_n\given Z_n), N(0,\sigma^2)\bigr)\ra0$ almost surely.  From the
assumption we know that every subsequence has a further subsequence with
$\E (e^{tY_n}\given Z_n)\ra e^{t^2\sigma^2/2}$ almost surely. For a countable set of $t$, we can
construct a single subsequence with this property for every $t$ by a diagonalization scheme. Part (i) gives that $d_{BL}\bigl(\L(Y_n\given Z_n), N(0,\sigma^2)\bigr)\ra0$ almost surely along
 this subsequence.

(iii) This is immediate from (i).
\end{proof}

\section{Contraction rates}\label{sec:contraction}

A first step in proving the semiparametric BvM is localizing the posterior near the true parameter
by establishing a contraction rate. In this section we achieve this for the Gaussian priors
in Section~\ref{SectionGaussianPriors}, mainly by combining the results of
\cite{ghosal2000,vandervaart2008} in our setting.

In the case of a product prior on the different components of $(a,b,f)$,
the posterior is a product measure since the likelihood \eqref{eq:likelihood_full}
factorizes. It then suffices to consider contraction in each component separately,
and when considering the posterior distribution of one component, the other
component(s) can be fixed to their true values, without loss of generality.

In particular, when considering an independent prior on the $b$ component, 
we can set the parameters $a$ and $f$ to $a_0$ and $f_0$ and incorporate these into the dominating measure.
Thus we use the restricted likelihood with parametrization $\eta^b= \Psi^{-1}(b)$ given by 
\begin{align*}
p_{\eta^b} (x) =  \Psi(\eta^b(z))^{ry} (1-\Psi(\eta^b(z)))^{r(1-y)} 
\end{align*}
with respect to the dominating measure
$$d\nu(z,r,y) = (1/a_0(z))^{r} (1-1/a_0(z))^{1-r} \, d\mu (r,y)\ dF_0(z),$$
where $\mu$ is counting measure on $\{\{0,0\}, \{1,0\}, \{1,1\} \}$.

The first result is an analogue of Lemma~3.2 of \cite{vandervaart2008}; the proof is similar
and omitted.

\begin{lemma}\label{lem:norm_ineq}
For any measurable functions $v^b,w^b:[0,1]\rightarrow \R$ and $q>1$,
\begin{itemize}
\item $\| p_{v^b} - p_{w^b} \|_{L^q(\nu)} = \|\Psi(v^b) - \Psi(w^b)\|_{L^q(F_0/a_0)} \leq  \|v^b-w^b\|_{L^q(F_0)}$,
\item $K(p_{v^b},p_{w^b}) \leq  \|v^b-w^b\|_{L^2(F_0)}^2$,
\item $V(p_{v^b}, p_{w^b}) \leq  \|v^b-w^b\|_{L^2(F_0)}^2$.
\end{itemize}
\end{lemma}

\begin{lemma}\label{lem:contrac_prod_b}
Consider the Gaussian process prior \eqref{eq:prior_b} for $b$. If $\varepsilon_n^b$ satisfies
\eqref{eq:gaus_conc_rate}, then the posterior distribution for $b$ concentrates about $b_0$ in $L^2(F_0)$ at rate
$\varepsilon_n^b$.
\end{lemma}

\begin{proof}
Since the densities $p_{w^b}$ are uniformly bounded, the $L^2(\nu)$-norm of the densities is bounded above by a multiple of the Hellinger distance. We can thus apply Theorem~2.1 of \cite{ghosal2000} with $d$ equal to the $L^2(\nu)$-norm. Using Lemma~\ref{lem:norm_ineq} and arguing as in Theorem~3.2 of \cite{vandervaart2008} gives contraction of the density in $L^2(\nu)$. Using the first result in Lemma~\ref{lem:norm_ineq}, this equals the $L^2(F_0/a_0)$-norm, which is equivalent to the $L^2(F_0)$-norm if $1/a_0$ is bounded away from zero.
\end{proof}

Next consider the propensity score-dependent prior \eqref{eq:prior_dep_EB} using the
external estimator $\hat a_n$.
As explained in the first paragraph of the proof of Theorem~\ref{thm:dep_prior_gen_EB}
it suffices to have the contraction rate with $\hat a_n$ set equal to  a deterministic 
sequence of functions satisfying the known restrictions of $\hat a_n$. Only boundedness
of the latter functions is needed.

\begin{lemma}\label{lem:contrac_dep_EB}
Consider the  prior $\Psi(W^b+\lambda a_n)$ for $b$ with  $W^b$ a centered
Gaussian process independent of $\lambda\sim N(0,\sigma_n^2)$ and $a_n$ a sequence of functions
with $\|a_n\|_\infty=O(1)$. If $\varepsilon_n^b$ satisfies \eqref{eq:gaus_conc_rate} and $\sigma_n=O(1)$, 
then the posterior for $b$ concentrates about $b_0$ in $L^2(F_0)$ at rate $\varepsilon_n^b$.
\end{lemma}

\begin{proof}
Since the densities $p_{w+\lambda a_n}$ are uniformly bounded, the $L^2(\nu)$-norm of the
densities is bounded above by a multiple of the Hellinger distance. We can thus apply a triangular
version of Theorem~2.1 of \cite{ghosal2000}  (which is contained in \cite{ghosal2007})
with $d$ equal to the $L^2(\nu)$-norm. 

Because $\|w + \lambda a_n- \eta_0^b\|_\infty \leq \|w-\eta_0^b\|_\infty + |\lambda| \|a_n\|_\infty$ and
$\|a_n\|_\infty$ is bounded by assumption, 
Lemma~\ref{lem:norm_ineq} gives the existence of a constant $c$ such that
\begin{align*}
\bigl\{ (w,\lambda):  \|w-\eta_0^b\|_\infty \leq c\varepsilon_n, |\lambda|\leq c\varepsilon_n \bigr\}
\subset \bigl\{ (w,\lambda): (K\vee V)(p_{\eta_0^b},p_{w+\lambda a_n}) \leq \varepsilon_n^2 \bigr\}.
\end{align*}
By the prior independence of $W^b$ and $\lambda$, the prior probability of the set on the
left is lower bounded by
$\Pi (\|W^b-\eta_0^b\|_\infty \le c \varepsilon_n) \,\Pi (|\lambda| \leq c \varepsilon_n)$.
For $\varepsilon_n=\varepsilon_n^b$ satisfying \eqref{eq:gaus_conc_rate}, the first 
term is lower bounded by $e^{-C_1n\varepsilon_n^2}$ in view of 
Theorem~2.1 of \cite{vandervaart2008}. The second term is  bounded below by a multiple of 
$(\varepsilon_n/\sigma_n) \wedge 1 \gtrsim \varepsilon_n$, for $\sigma_n = O(1)$.
This verifies (2.4) of Theorem~2.1 of \cite{ghosal2000}.

Let $B_n$ denote the sets constructed in Theorem~2.1 of \cite{vandervaart2008} 
for the Gaussian process $W^b$ and set, for some large enough $M>0$,
$$\mathcal{P}_n = \bigl\{ p_{w+\lambda a_n} : w\in B_n, |\lambda| \leq M\sigma_n \sqrt{n}\varepsilon_n \bigr\}.$$
Property (2.3) of Theorem~2.1 of \cite{vandervaart2008} gives that $\Pi(B_n^c) \leq e^{-Cn\varepsilon_n^2}$. 
Combined with the univariate Gaussian tail inequality for $\lambda$ and a union bound,
this yields that $\Pi(\mathcal{P}_n^c) \lesssim e^{-Cn\varepsilon_n^2}$. 
By the first assertion of Lemma~\ref{lem:norm_ineq} and the triangle inequality,
$\|p_{w+\lambda a_n} - p_{\bar{w}+\bar{\lambda}a_n}\|_{L^2(\nu)} 
\lesssim \|w-\bar{w}\|_{L^2(F_0)} + |\lambda-\bar{\lambda}|\, \|a_n\|_\infty$.
It follows that, for some $c_1>0$,
\begin{align*}
N(\mathcal{P}_n,\|\cdot\|_{L^2(\nu)},\varepsilon_n) 
\leq N(B_n,\|\cdot\|_\infty,\varepsilon_n/2) 
N\bigl([0,2M\sigma_n\sqrt{n}\varepsilon_n], |\cdot|, c_1\varepsilon_n) .
\end{align*}
Property (2.2) of Theorem~2.1 of \cite{vandervaart2008} gives that the logarithm of the first term
on the right side is bounded by a multiple of $n\varepsilon_n^2$. The logarithmic of the second term
grows at most logarithmically in $n$. Combined, these imply that
$\log N(\mathcal{P}_n,\|\cdot\|_{L^2(\nu)},\varepsilon_n) \lesssim n\varepsilon_n^2$.  This
concludes the proof of the verification of (2.2)-(2.4) in Theorem~2.1 of \cite{ghosal2000}, which
establishes posterior contraction in $L^2(\nu)$.  Using the first result in
Lemma~\ref{lem:norm_ineq}, this equals the $L^2(F_0/a_0)$-norm, which is equivalent to the
$L^2(F_0)$-norm since $1/a_0$ is assumed to be bounded away from zero.
\end{proof}


\begin{lemma}[Theorem~3.1 of \cite{vandervaart2008}]\label{lem:contrac_prod_f}
Consider the exponentiated Gaussian process prior \eqref{eq:prior_f} for $f$. If $\varepsilon_n^f$ satisfies the analogue of \eqref{eq:gaus_conc_rate} with $b$ replaced by
$f$, also in \eqref{EqConcentrationFunction}, then the posterior for $f$ concentrates about $f_0$ in Hellinger distance at rate $\varepsilon_n^f$.
\end{lemma}

\begin{lemma}[Proposition 5 of \cite{castillo2015}]\label{lem:hell_L2_rate}
Let $\log f_0\in C^{\gamma}([0,1]^d)$ and consider the exponentiated prior \eqref{eq:prior_f} for $f$. 
Suppose that the mean-zero Gaussian process $W^f$ takes values in $C^\delta([0,1]^d)$ for all $\delta<\bar{\gamma}$ and let $\eps_n^f$ satisfy the analogue of \eqref{eq:gaus_conc_rate} 
with $b$ replaced by $f$, also in \eqref{EqConcentrationFunction}. If for some $K_n \to \infty$ and some $0<\theta < \bar{\gamma}$,
\begin{equation}\label{eq:hell_L2_rate}
\eps_n^f K_n^{d/2} + \sqrt{n}\eps_n^f K_n^{-\theta} +K_n^{-\gamma} \to 0,
\end{equation}
then the posterior for $f$ concentrates about $f_0$ in $L^2$-distance at rate $\eps_n^f$.
\end{lemma}

\end{document}